\documentclass
[
    a4paper,
    DIV=11,
    abstracton
]
{scrartcl}
\usepackage{fullpage,authblk,amsmath,amssymb,amsthm,amsfonts}
\usepackage{bm, bbm, dsfont}
\usepackage{enumerate}
\usepackage{tikz}
\usepackage{pgfplots}
\usepackage{caption}
\usepackage{subcaption}
\usepgfplotslibrary{fillbetween}
\usepackage{url,mathtools}
\usepackage{footmisc}

\usepackage[pdffitwindow=true,
		   pdfstartview={FitH},
            plainpages=false,
            pdfpagelabels=true,
            pdfpagemode=UseOutlines,
            pdfpagelayout=SinglePage,
            bookmarks=false,
            colorlinks=true,
            hyperfootnotes=false,
            linkcolor=blue,
            urlcolor=blue!30!black,
            citecolor=green!50!black]{hyperref}

\newtheorem{theorem}{Theorem}[section]

\newtheorem{lem}[theorem]{Lemma}
\newtheorem{proposition}[theorem]{Proposition}

\newtheorem{assumption}[theorem]{Assumption}

\newtheorem{remark}[theorem]{Remark}

\theoremstyle{definition}
\newtheorem*{example*}{Example}

\numberwithin{equation}{section}

\newcommand{\R}{\mathbb{R}}

\newcommand{\abs}[1]{\left\vert#1\right\vert}
\newcommand{\Vrt}[1]{\left\Vert #1 \right\Vert}

\newcommand{\bb}[1]{\mathbb{#1}}

\newcommand{\Reals}{\mathbb{R}}
\newcommand{\Prob}{\mathbb{P}}

\newcommand{\const}{\mathrm{const.\,}}

\begin{document}
\title{Fr\'echet differentiable drift dependence of Perron--Frobenius and Koopman operators for non-deterministic dynamics}

\author{P\'eter Koltai}
\author{Han Cheng Lie}
\author{Martin Plonka\thanks{$\{$peter.koltai, hancheng.lie, martin.plonka$\}$@fu-berlin.de}}
\renewcommand\Affilfont{\small}
\affil[]{Department of Mathematics and Computer Science, Freie Universit\"at Berlin, Arnimallee 6, 14195 Berlin, Germany}

\date{}

\maketitle

\begin{abstract}
We prove the Fr\'{e}chet differentiability with respect to the drift of Perron--Frobenius and Koopman operators associated to time-inhomogeneous ordinary stochastic differential equations. This result relies on a similar differentiability result for pathwise expectations of path functionals of the solution of the stochastic differential equation, which we establish using Girsanov's formula. We demonstrate the significance of our result in the context of dynamical systems and operator theory, by proving continuously differentiable drift dependence of the simple eigen- and singular values and the corresponding eigen- and singular functions of the stochastic Perron--Frobenius and Koopman operators.
\end{abstract}



\section{Introduction}
In the study of dynamical systems and differential equations in particular, one important aspect is the sensitivity of the solution with respect to the data. In the literature, there exist classical results for the deterministic case that deal with dependence on the inital data and the driving velocity field. The same dependencies also arise for non-deterministic systems, e.g.~\cite{flandoli_etal,Andersson2017} for Fr\'{e}chet-type dependencies on the initial data, and \cite{Fournie1999,gobet_munos,monoyios,dieker_gao} for the dependence of path functionals or their expectations with respect to changes in the drift. This paper uses the approach of~\cite{Fournie1999}, which establishes a G\^{a}teaux-type dependence on the data by establishing the existence of directional derivatives with respect to the drift, in order to establish the Fr\'echet-type dependence of the solution operator with respect to an additive change of drift in a sufficiently smooth setting: for a suitable observable $g$, we provide in Theorem~\ref{thm:frechet_expectation} the Fr\'echet derivative at $\gamma =0$ of the non-linear functional
\begin{equation}\label{eq:func-u}
	u_g^{x} (\gamma) := \bb{E} \left[ g (X^{\gamma}) \,\big\vert\, X_0^{\gamma}=x \right] \nonumber
\end{equation}
with respect to additive perturbations in $\gamma$; above, $X^{\gamma}$ denotes the solution of the perturbed stochastic differential equation \eqref{eq:perturbed_sde} below. 

Our main motivation for considering the above result is the study of global long-term properties of dynamical systems, such as their stationary distribution, or the rate of mixing. 
Such properties are strongly related to spectral objects of the so-called \emph{transfer operators} associated to the dynamics~\cite{Pie94,SiEtAl08,SiSpAn09,FRGTSa16,FrSa17}.
These are infinite-dimensional linear operators describing the evolution of distributions and observables under the non-linear, stochastic dynamics. For instance, for a dynamical system given by the mapping~$\Phi:\Reals^d\to \Reals^d$, the \emph{Koopman operator}, acting on observables~$g: \Reals^d\to\Reals$, is given by~$\mathcal{U}g (x) = g\left(\Phi\left( x\right)\right)$.
If the dynamics is non-deterministic, e.g., the value of a stochastic process~$X$ at some time~$t$, then~$\mathcal{U}g (x) = \bb{E}[g(\Phi(x))] = \bb{E}^x[g(X_t)]$. If we now allow for non-deterministic initial conditions~$X_0$ with some distribution~$f$, i.e.,~$X_0 \sim f$, the \emph{Perron--Frobenius operator}~$\mathcal{P}$ is defined by~$\Phi(X_0) \sim \mathcal{P}f$. These two operators are adjoints on the appropriate spaces~\cite{LaMa94}, and allow for far-reaching analysis of the underlying dynamics; such analysis continues to be intensively exploited in applications such as (geophysical) fluid dynamics~\cite{FPET07,DFHPG09,FrLlSa10,FHRvS15,ArMe17}, molecular dynamics~\cite{DDJSch96,SFHD99,prinz2011markov,SchSa13,A19-1,BitEtAl17}, or stability and control~\cite{MaMeMo13, MaMe16, PrBrKu18}. In particular, smoothness or differentiability of spectral objects of the transfer operator might allow for more efficient optimization of fluid mixing processes~\cite{FRGTSa16}.

One of our main results in this paper is the differentiable dependence of these transfer operators on the time-dependent drift function, i.e., on the driving velocity field. As we consider the operators in their induced norm topology on~$L^2$-spaces, we will consider genuinely non-deterministic dynamics. This is because for \textit{deterministic} systems, these spaces are too ``large'' for the transfer operators associated with deterministic systems to depend even continuously on the drift with respect to the induced operator norm, as the following example shows.

\begin{example*}[Discontinuous drift-dependence for deterministic dynamics]
Let us consider the time-$t$ map~$\Phi$ of some deterministic differential equation, and the time-$t$ map~$\tilde{\Phi}$ of a slight perturbation of the previous differential equation. Let $\Phi$ and $\tilde{\Phi}$ be such that the associated Koopman operators~$\mathcal{U}: g\mapsto g\circ\Phi$ and~$\tilde{\mathcal{U}}: g\mapsto g\circ\tilde{\Phi}$ are well-defined on~$L^2(\Lambda)$, where~$\Lambda$ is the Lebesgue measure; cf.~\cite{LaMa94,Wal00} for details. For~$x$ with~$\Phi(x) \neq \tilde{\Phi}(x)$, let us consider the function sequence~$f_n = \Lambda(A_n)^{-1/2} \mathds{1}_{A_n}$, where $\{A_n\}_{n\in\mathbb{N}}$ is a sequence of balls with~$A_n \downarrow \{ \Phi(x) \}$, and~$\mathds{1}_A$ denotes the characteristic function of a set~$A$. Now, it is easy to see that there will be some~$N$ such that~$\| \mathcal{U}f_n - \tilde{\mathcal{U}}f_n \|_{L^2(\Lambda)} = 2$ for all~$n\ge N$, however small the perturbation in the drift is chosen to be. This shows that~$\mathcal{U}$ cannot depend continuously on the drift.
\end{example*}
The prior example implies that continuous differentiability cannot hold in the deterministic case. Combining this observation with duality arguments from Lemma~\ref{lem:chain_rule} and Remark~\ref{rem:chain_rule} shows that the same assertion also holds true for the Perron--Frobenius operator on~$L^2(\Lambda)$. Note that the main reason for non-continuous drift-dependence in the example is that functions with highly localized supports are mapped by $\mathcal{U}$ and~$\mathcal{P}$ to functions with highly localized supports. However, for non-deterministic systems driven by non-degenerate noise, the noise has the effect of spreading the support of the initial conditions, thus enabling a smooth drift-dependence of the associated transfer operators. We also note that, while the prior example rules out smooth drift-dependence of transfer operators associated with deterministic dynamics on~$L^p$-spaces, there could be other spaces where this property is retained~\cite{Bal18}. Such spaces, however, would require norms that ``punish'' increasingly localized densities~\cite{gouezel2006banach,Thi12}. In addition, single trajectories or realizations of both deterministic ordinary- and non-deterministic differential equations can be shown to depend smoothly on the drift, by using the integral form of the differential equation and the implicit function theorem. As this will not be of importance further on, we leave the details to reader.

We would like to remark that smooth dependence of invariant measures---often called \emph{linear response}---has been considered in specific cases. For instance, Butterley and Liverani~\cite{butterley2007smooth} show differentiability of SRB measures corresponding to Anosov flows with respect to one-dimensional parameters, and give an exhaustive report on the work that has been done previously in this field. For a survey on linear response results for deterministic systems (maps), we refer to~\cite{Bal14}. Results for non-deterministic systems arise in the context of random compositions of maps~\cite{BaRuSa17}, or for stochastic ordinary and partial differential equations in the weak topology~\cite{HaMa10}. The latter reference shows G{\^a}teaux-type pointwise differentiability of the transfer operators acting on smooth functions with respect to a real parameter, where also the (constant-in-space) diffusion matrix can vary with the parameter. It belongs to the class of approaches where functional-analytic tools are used to derive quantitative perturbation results for transfer operators and their dominant eigenmodes~\cite{KeLi99,Gal15,Sed18,GaGi19}. Our improvement over these results is that we consider (i) Fr\'echet differentiability in the infinite-dimensional space of velocity fields driving the stochastic differential equations (ii) of the associated transfer operator in its strong norm topology, and thus obtain differentiability of all isolated eigenvalues.

One question that we do not address in this paper is differentiability with respect to the diffusion matrix. In this paper, we rely on Girsanov's formula in order to obtain an expression for the Radon--Nikodym derivative between two path measures; this expression is crucial to our estimates for computing the Fr\'{e}chet derivative. Girsanov's formula applies in the case where the drift term is changed, and is a fundamental result in the theory of stochastic differential equations. To the best of our knowledge, an analogous formula for changes of the diffusion term does not exist. This is because the set of possible changes that one can apply to the diffusion matrix while still obtaining a mutually absolutely continuous path measure is severely constrained. Consider the simple case of a constant zero drift term and identity diffusion matrix. In this case, the path measure is Wiener measure on path space, which is a Gaussian measure. If one scales the identity diffusion matrix by a constant whose absolute value is not 1, then the resulting path measure is not mutually absolutely continuous with respect to Wiener measure \cite[Example 2.7.4]{Bogachev1998}. Thus, no Radon--Nikodym derivative exists, and the approach we follow for changes of the drift cannot be applied for such changes of the diffusion matrix, even in this simple case.

In this paper, we first introduce the setting which we are working in, state our assumptions, and introduce some key auxiliary results in Section~\ref{sec:prereq}. Then, we will extend ideas from~\cite{Fournie1999} to establish Fr\'echet differentiability of pathwise expectations with respect to the drift in Theorem~\ref{thm:frechet_expectation} of Section~\ref{sec:frechet_expectation}. Section~\ref{sec:transfer_ops} uses the quantitative remainder estimates from the preceding section to show our main result: differentiable drift-dependence of the transfer operators associated with non-deterministic dynamics governed by stochastic differential equations, Theorem~\ref{thm:PFO_dif}. This is then applied in Section~\ref{sec:EVD} to carry over the smooth dependence on the drift to isolated eigenvalues and corresponding eigenfunctions of transfer operators in Theorem~\ref{thm:spectrum_dif}. We conclude by pointing out two immediate consequences for dynamical systems. For convenience, we collect some relevant results on stochastic processes and transition kernels in the appendices.


\section{Prerequisites from stochastic analysis}
\label{sec:prereq}

\paragraph{Stochastic differential equations.}

For the rest of this work we are concerned with the following objects. 
We fix $T>0$ and~$\Omega:=C([0,T];\Reals^d)\coloneqq\{ f:[0,T]\to\Reals^d\, \vert\, f \text{ continuous}\}$. Further, let $\mathcal{F}$ be a $\sigma$-algebra on $\Omega$, and $\Prob$ be a probability measure on~$(\Omega,\mathcal{F})$. For a filtration~$(\mathcal{F}_t)_{t\in [0,T]}$ with $\mathcal{F}_t\subseteq \mathcal{F}$ for every $t\in [0,T]$, let~$W=(W_t,\mathcal{F}_t)_{t\in[0,T]}$ be the $d$-dimensional Wiener process with respect to~$\Prob$. We use the notation $W=(W_t,\mathcal{F}_t)_{t\in[0,T]}$ to emphasize that $W_t$ is $\mathcal{F}_t$-measurable for every $t\in [0,T]$. 
We define a measurable functional $h$ acting on $[0,T]\times \Omega$ to be \textit{adapted} to the filtration $(\mathcal{F}_t)_{t\in [0,T]}$ if $h(t,\cdot)$ is $\mathcal{F}_t$-measurable, for every $t\in [0,T]$. Let the \emph{drift}~$b:[0,T]\times \Reals^d\to\Reals^d$ and the \emph{diffusion matrix}~$\sigma:[0,T]\times \Reals^d\to GL(\R,d)$ (the set of real invertible $d \times d$ matrices) be measurable functions such that the following stochastic differential equation (SDE) admits a unique strong solution (cf. Theorem~\ref{thm:theorem4_6_cor}):
\begin{equation}\label{eq:original_sde}
	dX_t=b(t,X_t)dt+\sigma(t,X_t)dW_t,\quad X_0=x.
\end{equation}
Fix a deterministic initial condition~$x\in\Reals^d$. Let $\Prob^x$ denote the conditioning of $\Prob$ to the subset~$\{f:[0,T]\to\Reals^d \,\vert\, f(0)=x\}$ of~$\Omega$.
For simplicity, we shall assume that $b$ and $\sigma$ satisfy the conditions (A1) and (A2) (the Lipschitz continuity~\eqref{eq:A1-Lip} and linear growth conditions~\eqref{eq:A2-Growth}), which we recall below for convenience.
\begin{assumption} \label{ass:coefficient_bounds}~There exist constants $0<L,\lambda_\sigma<\infty$ that do not depend on $y$, $y'$, or $t$, such that 
	\begin{itemize}
		\item[(A1)]	The functions $b$ and $\sigma$ satisfy the Lipschitz conditions
		\begin{equation}\label{eq:A1-Lip}
			\abs{b_i(t,y)-b_i(t,y')}\leq L\abs{y-y'},\quad\abs{\sigma_{ij}(t,y)-\sigma_{ij}(t',y')}\leq L \abs{y-y'}
		\end{equation}
		for all $(t,y,y')\in [0,T]\times \Reals^d\times\Reals^d$ and $1\leq i,j\leq d$.
		\item[(A2)]	The functions $b$ and $\sigma$ satisfy the growth conditions
		\begin{equation}\label{eq:A2-Growth}
			\abs{b_i(t,y)}\leq L(1+\abs{y}),\quad \abs{\sigma_{ij}(t,y)}\leq L(1+\abs{y})
		\end{equation}
		for all $(t,y,y')\in [0,T]\times \Reals^d\times\Reals^d$ and $1\leq i,j\leq d$.
		\item[(A3)]	The function $\sigma$ satisfies
		\begin{equation}\label{eq:uniform_ellipticity}
\lambda_\sigma^{-2}\abs{\xi}^2\leq \xi^\top \sigma\sigma^\top(t,y)\xi, 
		\end{equation}
		for all $(t,y,\xi)\in[0,T]\times\Reals^d\times (\Reals^d\setminus\{0\})$, where $\xi^\top$ denotes the transpose of $\xi$.
		\item[(A4)]It holds that
		\begin{equation}\label{eq:P_bound}
 			\bb{P}^x\left(\int_0^T \left(\abs{\sigma^{-1}_s(b_s+\gamma_s)}^2+\abs{\sigma^{-1}_sb_s}^2\right)\left(X^0_s\right)ds<\infty\right)=1.
		\end{equation}
		(Here and in the following, the subscript `$s$' denotes the time input.)
	\end{itemize}
\end{assumption}
\begin{remark}
	The last condition (A4) can be derived from the other assumptions, even under the more general condition of locally Lipschitz drift and diffusion matrix. However, for the sake of simplicity, in Lemma~\ref{lem:A4} a proof is given for globally Lipschitz drift, change of drift, and diffusion.
\end{remark}

\paragraph{Norms for the drift.}

Let us consider the vector space 
\begin{multline}
	\mathcal{C}\coloneqq\left\{f:[0,T]\times\Reals^d\to\Reals^{n}\; \middle| \;  \exists \, 0<L<\infty\text{ such that }\abs{f_i(t,y)-f_i(t,y')}\leq L\abs{y-y'}\right. \label{eq:class_C}
	\\
	\left.\text{ and } \abs{f_i(t,y)}\leq L(1+\abs{y}),\ \forall (t,y)\in [0,T]\times\bb{R}^d,\ 1\leq i\leq n\right\} \, .
\end{multline}
Above, we allow $n\in\bb{N}$ to vary, so that by using the `stacked' vector representations of matrices, we can consider $\cal{C}$ to contain both vector- and matrix-valued functions whose entries satisfy the Lipschitz continuity and growth conditions. In particular, the drift $b$ and diffusion $\sigma$ introduced earlier belong to $\cal{C}$. The essential idea is that $\cal{C}$ contains functions from $[0,T]\times \bb{R}^d$ to some Euclidean space such that each component of the image is Lipschitz continuous and grows linearly. In particular, if $\gamma \in\mathcal{C}$, then
\begin{equation}\label{eq:perturbed_sde}
	dX^\gamma_t=[b(t,X^\gamma_t)+\gamma(t,X^\gamma_t)]dt+\sigma(t, X^\gamma_t)dW_t,\quad X^\gamma_0=x,
\end{equation}
admits a unique strong solution $X^\gamma$, by Theorem~\ref{thm:theorem4_6_cor}. Using this notation, we have that $X^0$ denotes the unique strong solution of the unperturbed SDE~\eqref{eq:original_sde}.

\textbf{Stochastics notation}. Let $\bb{E}^{x}$ denote the expectation operator with respect to $\Prob^x$, and let $\mu^{\gamma,x}\coloneqq \Prob^x\circ (X^\gamma)^{-1}$ denote the \textit{law of the solution} of~\eqref{eq:perturbed_sde}. Using this notation, it follows that~$\mu^{0,x}$ is the law of the solution of the unperturbed SDE~\eqref{eq:original_sde}. We shall write $\bb{E}^{\gamma,x}$ ($\bb{E}^{0,x}$) to denote the expectation with respect to $\mu^{\gamma,x}$ ($\mu^{0,x}$). Thus, if $g:\Omega\to\bb{R}$ is a path functional, the shorthand~$\bb{E}^{0,x}\left[ g \right]$ denotes~$\bb{E}^x\left[ g\left(X^0\right) \right]$.

We shall say that $f:[0,T]\times\Reals^d\to\Reals^d$ is $\mu^{0,x}$-\emph{almost surely bounded} if there exists some $0<C<\infty$ such that
\begin{equation}\label{eq:almost_surely_bounded_changes_of_drift}
	\bb{P}^x\left(\abs{f(s,X^0_s)}\leq C,\ \forall s\in [0,T]\right)=1
\end{equation}
holds. Utilizing this notion we set 
\begin{equation*}
	L^\infty(\mu^{0,x})\coloneqq \{f:[0,T]\times\Reals^d\to\Reals^d \,\big\vert\, f \text{ satisfies \eqref{eq:almost_surely_bounded_changes_of_drift} for some $0<C<\infty$}\}.
\end{equation*}
For every $f\in L^\infty(\mu^{0,x})$, we shall write 
\begin{equation*}
	\Vrt{f}_{L^\infty(\mu^{0,x})}\coloneqq\min\left\{C>0\, \middle\vert\,  \text{\eqref{eq:almost_surely_bounded_changes_of_drift} holds}\right\} \, .
\end{equation*}
We shall omit the dependence on the parameter $\mu^{0,x}$ and simply write $\Vrt{f}_{L^\infty}$ below. Note that $(L^{\infty}(\mu^{0,x}),\| \cdot \|_{L^{\infty}})$ is a normed vector space. 

Define the function class~$ V:= \mathcal{C}\cap L^\infty(\mu^{0,x})$, and equip $V$ with the following norm:
\begin{multline}
	\Vrt{f}_V\coloneqq\min\left\{C>0\ \right\vert\ \abs{f_i(t,y)}\leq C, \ \abs{f_i(t,y)-f_i(t,y')}\leq C\abs{y-y'},
 \\
 \left.\ \text{for all } 1\leq i\leq n,\ \text{for all } (t,y,y')\in[0,T]\times\bb{R}^d\times\bb{R}^d\right\}.\label{eq:norm_on_class_V}
\end{multline}
In Lemma~\ref{app:C-vec}, we show that $(V,\Vrt{\cdot}_V)$ is a Banach space. 

\paragraph{Key auxiliary results.}

Now, let~$x\in\Reals^d$ be arbitrary, and let us consider the map $u_g^x: V \to\Reals$, for a suitable observable~$g$ specified below in Theorem~\ref{thm:frechet_expectation} (C2), defined by
\begin{equation}\label{eq:expectation_of_path_functional_as_function_of_gamma}
	u_g^x(\gamma)\coloneqq\bb{E}^{x} \left[g(X^\gamma)\right].
\end{equation}
We wish to establish that $u_g^x$ is a Fr\'{e}chet differentiable map from $(V,\Vrt{\cdot}_V)$ to $\left( \R , |\cdot | \right)$. From now on, we will consider $\gamma \in V$.

First, we recall a result from~\cite{liptsershiryaevpartI} which will be crucial for the following steps.
Note that in the result below, we allow for the Wiener process $W$ to be of different dimension than the diffusion processes $\xi$ and $\eta$. The result below is adapted from~\cite[Section~7.6.4]{liptsershiryaevpartI}.
\begin{remark} \label{rem:reference_measure}
The main purpose of Theorem~\ref{thm:girsanovs_theorem} is that it allows us to express probabilistic objects---like expectation values---involving the process~$X^{\gamma}$ with respect to the law of the unperturbed process~$X^0$. This will be key to establishing the differentiability of~\eqref{eq:expectation_of_path_functional_as_function_of_gamma} with respect to~$\gamma$.
\end{remark}

\begin{theorem}\label{thm:girsanovs_theorem}
	Let $\bb{X}\subset \bb{R}^n$, and let $\xi=(\xi_t,\mathcal{F}_t)_{t\in[0,T]}$ and $\eta=(\eta_t,\mathcal{F}_t)_{t\in[0,T]}$ be $\bb{X}$-valued processes that satisfy
	\begin{align}
		d\xi_t&=\tilde{b}_t(\xi)dt+\sigma_t(\xi)dW_t,\quad \xi_0=\eta_0,
  		\label{eq:7_129}
  		\\
  		d\eta_t&=b_t(\eta)dt+\sigma_t(\eta)dW_t,
  		\label{eq:7_130}
	\end{align}
	where $W=(W_t,\mathcal{F}_t)_{t\in[0,T]}$ is a $k$-dimensional Wiener process, $\eta_0$ is a $\mathcal{F}_0$-measurable random variable with $\mathbb{P}(\sum_{i=1}^{n}\abs{(\eta_0)_i}<\infty)=1$, $b$ and $\tilde{b}$ are measurable functions from $[0,T]\times \mathbb{X}$ to $\Reals^n$, and $\sigma$ is a measurable function from $[0,T]\times\mathbb{X}$ to $\Reals^{n\times k}$, such that the following assumptions are fulfilled:
\begin{enumerate}
	\item[(B1)]\label{girsanov_theorem_cond_I}  The system of algebraic equations
	\begin{equation*}
		\sigma_t(y)\alpha_t(y)=\tilde{b}_t(y)-b_t(y)
	\end{equation*}
	has a solution $\alpha_t$ for all $(t,y)\in [0,T]\times \mathbb{X} $.
	\item[(B2)] \label{girsanov_theorem_cond_II} The components of the functionals $(b_t(\cdot))_{t\in[0,T]}$, $(\tilde{b}_t(\cdot))_{t\in[0,T]}$, and $(\sigma_t(\cdot))_{t\in[0,T]}$ are adapted to the filtration $(\mathcal{F}_t)_{t\in [0,T]}$ and satisfy (A1) and (A2), so that both \eqref{eq:7_129} and \eqref{eq:7_130} have unique strong solutions $\xi$ and $\eta$ respectively.
	\item[(B3)] \label{girsanov_theorem_cond_III} It holds that 
	\begin{equation}\label{eq:7_137}
		\Prob\left(\int_0^T\left( \tilde{b}^\top_t(\sigma_t\sigma_t^\top)^+\tilde{b}_t+b^\top_t(\sigma_t\sigma_t^\top)^+b_t\right)(\xi)dt<\infty\right)=1, \nonumber
	\end{equation}
	where $(\sigma_t\sigma_t^\top)^+$ denotes the pseudoinverse of the $n\times n$ matrix $\sigma_t \sigma_t^\top$.  
\end{enumerate}
Then, the law $\mu_\xi$ of $\xi$ is absolutely continuous with respect to the law $\mu_\eta$ of $\eta$, and
	\begin{equation}\label{eq:7_138}
		\left.\frac{d\mu_\xi}{d\mu_\eta}\right\vert_{\mathcal{F}_t}(\eta)=\exp\left(\int_0^t (\tilde{b}_s-b_s)^\top (\sigma_s\sigma_s^\top)^+(\eta)d\eta_s-\frac{1}{2}\int_0^t (\tilde{b}_s-b_s)^\top (\sigma_s\sigma_s^\top)^+(\tilde{b}_s-b_s)(\eta)ds\right).
	\end{equation}
\end{theorem}
\begin{remark} \label{rem:B_conditions}
Following the steps of Lemma~\ref{lem:A4} one can prove that (B3) can be derived from (B2) and~(A3). Further, since in our case $b_t(X) = b(t,X_t)$, it follows that $(b_t)_{t\in [0,T]}$ is immediately adapted to $(\mathcal{F}_t)_{t\in [0,T]}$, and thus Assumption~\ref{ass:coefficient_bounds} implies~(B2); the same applies for $\tilde{b}_t := b_t + \gamma_t$ and~$\sigma$. Finally, as we assume $\sigma(t,x)$ to be invertible for all~$(t,x)$, (B1) is satisfied as well.
\end{remark}
From Girsanov's formula \eqref{eq:7_138} and its specification \eqref{eq:Z_gamma_t2} to our case below, it is known that the term inside the exponential is described by a continuous semimartingale
\begin{equation}\label{eq:M_gamma}
	M^\gamma_t\left(X^0\right)\coloneqq \int_0^t\gamma^\top(\sigma\sigma^\top)^{-1}(s,X^0_s)dX^0_s,\quad 
\end{equation}
and the associated quadratic variation process
\begin{equation}\label{eq:M_gamma_q}
	\langle M^\gamma\rangle_t\left(X^0\right)\coloneqq \int_0^t\abs{\sigma^{-1} \gamma(s,X^0_s)}^2ds \,.
\end{equation}
For an easier readability of the proof of our main result below, we first state two technical lemmas. We defer the proofs to the Appendix~\ref{app:proof-lem2}.

\begin{lem}\label{lem:Z_L2}
	The exponential martingale $(Z^{\gamma}_t(X^0))_{t\in [0,T]}$ described by
	\begin{align}
		Z^\gamma_t\left(X^0\right) & \coloneqq\left.\frac{d\mu^{\gamma ,x}}{d\mu^{0,x}}\right\vert_{\mathcal{F}_t}\left(X^0\right)= \exp \left(M_t^{\gamma} (X^0) - \frac{1}{2} \langle M^{\gamma} \rangle_t (X^0) \right) \nonumber \\
		&=\exp\left(\int_0^t\gamma^\top(\sigma\sigma^\top)^{-1}(s,X^0_s)dX^0_s-\frac{1}{2}\int_0^t\abs{\sigma^{-1} \gamma(s,X^0_s)}^2ds \right)\label{eq:Z_gamma_t2}
	\end{align}
	is square integrable with respect to $\mathbb{P}^x$.
\end{lem}

\begin{remark}
	Note that we choose \eqref{eq:original_sde} to correspond to \eqref{eq:7_130} and \eqref{eq:perturbed_sde} to correspond to \eqref{eq:7_129}, so that $\gamma$, $\mu^{0,x}$, and $\mu^{\gamma,x}$ in \eqref{eq:Z_gamma_t2} correspond to $\tilde{b}-b$, $\mu_\eta$, and $\mu_\xi$ in~\eqref{eq:7_138}, respectively.
\end{remark}
The following result will allow us to bound the series expansion of $Z^{\gamma}_t$ from above and pass to the limit $\gamma \rightarrow 0$ for the differentiability result.
\begin{lem}\label{lem:tech2}
	Under Assumption \ref{ass:coefficient_bounds} with the definitions~\eqref{eq:M_gamma} and~\eqref{eq:M_gamma_q} above, the following estimate holds:
	\begin{equation}
		\bb{E}^{0,x} \left[ \left( \sum_{n=2}^{\infty} \frac{1}{n!} \left( |M^{\gamma}_t| + \frac{1}{2} \langle M^{\gamma} \rangle_t \right)^n \right)^2 \right] \leq 2C \big(\exp \left( \lambda_\sigma \sqrt{t} \| \gamma \|_{L^{\infty}} \right)-1 -2 \lambda_\sigma\sqrt{t} \| \gamma \|_{L^{\infty}}\big)^2 \, \nonumber
	\end{equation}
	for $0\le t\le T$.
\end{lem}


\section{Fr\'{e}chet differentiability of expected path functionals}
\label{sec:frechet_expectation}

In order to prove the Fr\'echet differentiability of the non-linear function $u_g^x(\gamma) = \bb{E}^x\left[ g\left(X^\gamma\right) \right]$ with respect to~$\gamma$ at~$\gamma = 0$ we propose a derivative \eqref{eq:frechet_derivative} and show the required convergence. We do so by first shifting all the $\gamma$-dependence to a new stochastic process~$Z^{\gamma}$ by Theorem~\ref{thm:girsanovs_theorem} (see also Remark~\ref{rem:reference_measure}) and then employing Lemma~\ref{lem:Z_L2} and Lemma~\ref{lem:tech2} to pass to the limit.
\begin{theorem}\label{thm:frechet_expectation}
	Let $x\in\Reals^d$. Suppose that Assumption~\ref{ass:coefficient_bounds} above and following hold true:
	\begin{itemize}
		\item[(C1)] \label{suffcon_existence_soln_orig_sde_invertible_sigma} $b:[0,T]\times\Reals^d\to\Reals^d$ and $\sigma:[0,T]\times\Reals^d\to GL(\Reals,d)$ are elements of $\mathcal{C}$.
		\item[(C2)] \label{square_integrability_phi} \label{phi_L_1} $g:\Omega\to \mathbb{R}$ satisfies $g\in L^2(\mu^{0,x}) \cap L^1(\mu^{\gamma , x})$, for all $\gamma \in V= \mathcal{C}\cap L^\infty(\mu^{0,x})$.
	\end{itemize}
	Then the Fr\'{e}chet derivative of $u_g^x:V\to\Reals$ at $\gamma =0$ exists, and is given by
	\begin{equation}\label{eq:frechet_derivative}
		\left. D_{\beta}u_g^x\right\vert_{\beta=0}(\gamma)\coloneqq \bb{E}^x\left[g(X^0)\int_0^T\gamma^\top(\sigma\sigma^\top)^{-1}(s,X^0_s)dX^0_s\right]\equiv\bb{E}^{0,x}\left[ gM^\gamma_T\right].
	\end{equation}
\end{theorem}

\begin{remark}[Admissible observables] \label{rem:admissible_observable}
Before proceeding to the proof, let us examine which observables~$g$ satisfy condition~(C2) above.
\begin{enumerate}[(i)]
\item
If $\gamma\in V$, then $\gamma$ is almost surely bounded on $[0,T]\times\bb{R}^d$ and continuous, hence bounded everywhere. Therefore, by reversing the roles of $\tilde{b}$ and $b$ in Theorem~\ref{thm:girsanovs_theorem}, we obtain that~$\mu^{0,x}$ and $\mu^{\gamma,x}$ are in fact mutually locally equivalent. Thus, since both $\mu^{0,x}$ and $\mu^{\gamma,x}$ are probability (in particular, finite) measures,~$g\in L^{\infty}(\mu^{0,x})$ satisfies~(C2).

\item
In Sections \ref{sec:transfer_ops} and \ref{sec:EVD}, we shall focus on observables of the kind~$g = \tilde{g}\circ\pi_t$, where~$\tilde{g}:\Reals^d \to \Reals$ and~$\pi_t: \Omega \to \Reals^d$ is given by~$\pi_t(X) = X_t$, the coordinate projection for some~$0\le t \le T$. For such observables, we have
\begin{equation} \label{eq:marginal_norm}
\left\| g \right\|^2_{L^2(\mu^{0,x})} = \int_\Omega \tilde{g} \left( \pi_t \left(X \right) \right)^2\, d\mu^{0,x} = \int_{\Reals^d} \tilde{g}(y)^2\, d \left( \pi_t^{\#}\mu^{0,x} \right)(y)\,,
\end{equation}
where~$\pi_t^{\#}$ denotes the push-forward by~$\pi_t$, i.e., $\pi_t^{\#}\mu^{0,x} = \mu^{0,x} \circ \pi_t^{-1}$. The push-forward of the path-measure~$\mu^{0,x}$ is the distribution of the process~$X^0$ at time~$t$, i.e.,~$\pi_t^{\#}\mu^{0,x}(dy) = k_t(x,y,0)dy$, where the transition kernel~$k_t$ is introduced in~\eqref{eq:kernel} below. Similarly as in case~(i),~$g$ satisfies the condition (C2) for any~$\tilde{g} \in L^{\infty}(\Reals^d, \Lambda)$, with~$\Lambda$ being the Lebesgue measure.
\end{enumerate}
\end{remark}

\begin{proof}[Proof of Theorem \ref{thm:frechet_expectation}]
	By Remark~\ref{rem:B_conditions}, applying Theorem~\ref{thm:girsanovs_theorem} yields the formula \eqref{eq:Z_gamma_t2} for the Radon--Nikodym derivative of $\mu^{\gamma ,x}$ with respect to~$\mu^{0,x}$. Given the definition~\eqref{eq:expectation_of_path_functional_as_function_of_gamma} of the map $u_g^x$, it follows that
	\begin{equation*}
		u_g^x(\gamma)-u_g^x(0)=\bb{E}^x \left[g\left(X^0\right)\left(Z^\gamma_T\left(X^0\right)-1\right)\right]\equiv\bb{E}^{0,x}\left[(Z^\gamma_T-1)g\right],
	\end{equation*}
	where in the last equality we used the notation $\bb{E}^{0,x}$ to denote the expectation of functionals of $X^0$, or equivalently the expectation with respect to $\mu^{0,x}$. The square integrability of $Z^{\gamma}_t$ with respect to $\mu^{0,x}$ follows from Lemma~\ref{lem:Z_L2}. 
	
	We now show that $\bb{E}^{0,x}\big[\abs{Z^\gamma_T-1-M^\gamma_T}^2\big]$ converges to zero as $\Vrt{\gamma}_{L^\infty}$ converges to zero.  Using the definition of $Z^\gamma_T$, the series expansion of the exponential, the triangle inequality, and the inequality $(a+b)^2\leq 2(a^2+b^2)$, we have 
	\begin{equation}\label{eq:ineq01}
		\abs{Z^\gamma_T-1-M^\gamma_T}^2\leq \left(\frac{1}{2}\langle M^\gamma\rangle_T^2+ 2\left[\sum^\infty_{n=2}\frac{1}{n!}\left(M^\gamma_T-\frac{1}{2}\langle M^\gamma\rangle_T\right)^n\right]^2 \right)
	\end{equation}
	uniformly in~$x$. By Lemma~\ref{lem:tech2}, we have the following estimate
	\begin{align}
		\bb{E}^{0,x} \left[ \left( \sum_{n=2}^{\infty} \frac{1}{n!} \left( |M^{\gamma}_T| + \frac{1}{2} \langle M^{\gamma} \rangle_T \right)^n \right)^2 \right] &\leq 2C (\exp ( \lambda_\sigma\sqrt{T} \| \gamma \|_{L^{\infty}} )-1 -2 \lambda_\sigma\sqrt{T} \| \gamma \|_{L^{\infty}})^2 \nonumber \\
		 	&= 2C \sum_{\ell=4}^{\infty} \sum_{\substack{j,k \geq 2\\j+k = \ell}} \dfrac{1}{k!j!} (2\lambda_\sigma\sqrt{T})^l \| \gamma \|_{L^{\infty}}^{l}\; . \label{eq:frechet-estimate}
	\end{align}
	Together with~\eqref{eq:almost_sure_bound_on_qv}, an estimate for $\langle M^{\gamma}\rangle_T^2$, we obtain
	\begin{equation}\label{eq:bound}
		\frac{\bb{E}^{0,x}\left[\abs{Z^\gamma_T-1-M^\gamma_T}^2\right]}{\Vrt{\gamma}^2_{L^\infty}}\leq \frac{1}{2}(\lambda_\sigma\sqrt{T})^4\Vrt{\gamma}^2_{L^\infty}+2C\sum_{\ell=4}^\infty\sum_{\substack{j,k\geq 2\\ j+k=\ell}}\frac{1}{j!k!}\left(2\lambda_\sigma\sqrt{T}\right)^{\ell}\Vrt{\gamma}_{L^\infty}^{ \ell-2},
	\end{equation}
	where the right-hand side is $O(\Vrt{\gamma}^2_{L^\infty})$, and hence decreases to zero as $\Vrt{\gamma}_{L^\infty}$ decreases to zero. Note that the parameter $\lambda_\sigma\sqrt{T}$ determines the convergence rate.
	
	To complete the proof, we observe that
	\begin{equation*}
		u_g^x(\gamma)-u_g^x(0)-\bb{E}^{0,x} \left[g M^\gamma_T\right]=\bb{E}^{0,x} \left[\left(Z^\gamma_T-1-M^\gamma_T\right)g\right]
	\end{equation*}
	together with the Cauchy--Schwarz inequality implies
	\begin{equation}\label{eq:lim_frechet1}
		\frac{\abs{ u_g^x(\gamma)-u_g^x(0)-\bb{E}^{0,x} \left[g M^\gamma_T\right]}}{\Vrt{\gamma}_{L^\infty}}\leq \bb{E}^{0,x}\left[\abs{g}^2\right]^{1/2}\left(\frac{\bb{E}^{0,x}\left[\abs{Z^\gamma_T-1-M^\gamma_T}^2\right]}{\Vrt{\gamma}^2_{L^\infty}}\right)^{1/2}.
	\end{equation}
	Finally,~\eqref{eq:bound} gives
	\begin{equation}\label{eq:lim_frechet2}
		\lim_{\Vrt{\gamma}_{L^\infty}\to 0}\frac{\bb{E}^{0,x}\left[\abs{Z^\gamma_T-1-M^\gamma_T}^2\right]}{\Vrt{\gamma}^2_{L^\infty}}=0,
	\end{equation}
	which shows that the Fr\'{e}chet derivative of $u_g^x$ at zero exists and is given by~$\bb{E}^{0,x}[g M^\gamma_T]$.

	Given the definition \eqref{eq:norm_on_class_V} of $\| \cdot \|_V$, it follows that if $\gamma\in V$, then $\|\gamma \|_{L^{\infty}}\leq  \| \gamma \|_V$. Thus, by \eqref{eq:frechet-estimate}, we have
	\begin{equation}
		\lim_{\| \gamma \|_V \rightarrow 0}\dfrac{\bb{E}^{0,x}\left[\abs{Z^\gamma_T-1-M^\gamma_T}^2\right]}{\Vrt{\gamma}^2_{V}}=0, \nonumber
	\end{equation}
	as desired.
\end{proof}
\begin{remark}[Continuity of the differential] \label{rem:derivative_continuous}
Instead of calling the above ``differentiation in $\beta = 0$'' we could also call it ``differentiation in $\beta = b$'' and denote~$X^0$ by~$X^b$ instead. We chose~$0$ here for simplicity and the interpretation of being unperturbed. The different notation would lead to the following representation of the derivative
	\begin{equation}\label{eq:derivative_at_b}
	\left. D_{\beta}u_g^x\right\vert_{\beta=b}(\gamma)\coloneqq \bb{E}^x\left[g(X^b)\int_0^T \gamma^\top(\sigma\sigma^\top)^{-1}(s,X^b_s)dX^b_s\right]\equiv\bb{E}^{b,x}\left[ gM^\gamma\right]. \nonumber
	\end{equation}
It also poses the question whether we have continuity in~$b$ and thus continuous differentiability everywhere. This is indeed the case as the following proposition shows.
\end{remark}

\begin{proposition}
	Let us assume that $gM^\gamma$ belongs to $L^2(\mu^{b,x})$ and that the assumptions of Theorem ~\ref{thm:frechet_expectation} hold for $X^0$ replaced by $X^b$. Then the map
	\begin{equation}
		b \mapsto D_{\beta} u_g^x |_{\beta = b} ( \cdot )
	\end{equation}
	is continuous in $b$.
\end{proposition}
\begin{proof}[Sketch of proof]
Under the stated hypotheses, one way to show continuous differentiability everywhere is to use the same strategy as for the proof of Theorem \ref{thm:frechet_expectation}. By~\eqref{eq:derivative_at_b},
\begin{align*}
 & \left. D_{\beta}u_g^x\right\vert_{\beta=b}(\gamma)-\left. D_{\beta}u_g^x\right\vert_{\beta=b+b'}(\gamma)
 \\
 &= \bb{E}^x\left[g(X^b)\int_0^T \gamma^\top(\sigma\sigma^\top)^{-1}(s,X^b_s)dX^b_s\left(1-\frac{d\mu^{b+b',x}}{d\mu^{b,x}} \right)(X^b)\right]
 \\
 &= \bb{E}^x\left[g(X^b)\int_0^T \gamma^\top(\sigma\sigma^\top)^{-1}(s,X^b_s)dX^b_s\left(1-Z^{b'}_T(X^b)\right)\right]\equiv \bb{E}^{b,x}\left[ gM^\gamma(1-Z^{b'}_T)\right].
\end{align*}
Therefore,
\begin{enumerate}[(i)]
\item if $g M^\gamma$ belongs to $L^2(\mu^{b,x})$, and
\item   if $\displaystyle 
\lim_{\| b'\|_V\to 0} \bb{E}^x\left[\left\vert 1-Z^{b'}_T(X^b)\right\vert^2\right]=0$,
\end{enumerate}
then continuity of the derivative for fixed~$\gamma$ will follow. Note that assumptions (C1) and (C2) in Theorem~\ref{thm:frechet_expectation} readily imply~(i).

To show that (ii) is satisfied, we modify the steps in the proof of Theorem~\ref{thm:frechet_expectation}, by using \eqref{eq:ineq01} with~\eqref{eq:doobs_inequality} and~\eqref{eq:almost_sure_bound_on_qv} to bound~$\vert 1- Z^{b'}_T\vert^2$. Uniformity of the estimates in~$\|\gamma\|_V$ imply the claim.
\end{proof}

\begin{remark}[Time-$t$-observables] \label{rem:time_t_observable}
For a fixed but otherwise arbitrary~$ 0\le t \le T$, let us examine the structure of the derivative~\eqref{eq:frechet_derivative} for the observables~$\tilde{g}\circ\pi_t$ discussed in Remark~\ref{rem:admissible_observable}~(ii).

As~$\pi_t$ is $\mathcal{F}_t$-measurable,~$g = \tilde{g}\circ\pi_t$ is also $\mathcal{F}_t$-measurable, and as in the proof of Theorem~\ref{thm:frechet_expectation} we have
\begin{eqnarray*}
u^x_g(\gamma) = \bb{E}^x\left[ g\left( X^\gamma \right) \right]
= \int_{\Omega} g\, d\mu^{\gamma,x} 
& = & \int_{\Omega} g\, d\mu^{\gamma,x}\big\vert_{\mathcal{F}_t} \\
& \stackrel{\eqref{eq:7_138}, \eqref{eq:Z_gamma_t2}}{=} & \int_{\Omega} g Z^\gamma_t\, d\mu^{0,x}\big\vert_{\mathcal{F}_t} \\
& = &  \int_{\Omega} g Z^\gamma_t\, d\mu^{0,x} \\
& = & \bb{E}^x\left[ g\left( X^0 \right) Z^\gamma_t \left( X^0 \right)\right]
= \bb{E}^x\left[ \tilde{g}\left( X^0_t \right) Z^\gamma_t \left( X^0 \right) \right],
\end{eqnarray*}
where we used in the third and fifth equalities that $\int g\,d\mu = \int g\,d\mu\vert_{\mathcal{F}_t}$ for~$\mathcal{F}_t$-measurable~$g$. This leads to
\begin{equation*}
\left. D_{\beta}u_g^x\right\vert_{\beta=0}(\gamma) = \bb{E}^x\left[ \tilde{g} \left(X^0_t \right)\int_0^t \gamma^\top(\sigma\sigma^\top)^{-1}(s,X^0_s)dX^0_s \right]\equiv\bb{E}^{0,x}\left[(\tilde{g}\circ \pi_t)M^\gamma_t\right].
\end{equation*}
Note the difference in the equation above with  respect to~\eqref{eq:frechet_derivative}: the observable~$\tilde{g}$ now does not depend on the entire path~$X^0$, only on its value at time~$t$, and the It\^o integral goes up to~$t$ instead of up to~$T$.
\end{remark}

The preceding results extend to the case where we augment the SDE \eqref{eq:original_sde} with normal reflecting boundary conditions for some smooth, bounded domain~$\mathbb{X}$. In this case, the reflecting boundary conditions are encoded by the local time process associated to the boundary; see, e.g.,~\cite[Chapter IX, \S 2, Exercise 2.14]{revuz_yor} for a one-dimensional example, or~\cite[Theorem 2.4.1]{Pil14} for the general case. Recall that Girsanov's formula describes the Radon--Nikodym derivative of two mutually equivalent probability measures $\tilde{\bb{P}}$ and~$\bb{P}$. Since the boundary conditions are invariant under changes between mutually equivalent probability measures, the preceding analysis carries over to the case of diffusions with reflection.


\section{Fr\'{e}chet differentiability of transfer operators}
\label{sec:transfer_ops}
We will utilize in this section our results from above to show the differentiability of the Perron--Frobenius operator and the Koopman operator with respect to the drift.

We shall consider a compact subset~$\mathbb{X} \subset \R^d$ with $C^{1+\delta}$ boundary for some~$\delta>0$. In addition to Assumption~\ref{ass:coefficient_bounds}, let~$\sigma\in C^{2+\alpha}([0,T]\times\mathbb{X}; \Reals^{d\times d})$ and~$b, \gamma \in C^{1+\alpha}([0,T]\times\mathbb{X}; \Reals^d)$ for some~$\alpha>0$. Let us consider the processes~$X^0$ and~$X^\gamma$, governed by~\eqref{eq:original_sde} and~\eqref{eq:perturbed_sde} on~$\mathbb{X}$ with reflecting boundary conditions~\cite{Pil14}, respectively. For the rest of this section, let us fix some~$t\in (0,T]$.

Recall from Remark~\ref{rem:admissible_observable} that~$X_t^{\gamma}$, the process~$X^{\gamma}$ at time~$t$, has the distribution~$\mu_t^{\gamma,x} := \pi_t^{\#}\mu^{\gamma,x}$, where~$\pi_t^{\#}$ is the push-forward operator associated with the coordinate projection (``time-$t$-evaluation functional'')~$\pi_t$. Under the given assumptions, the distribution is absolutely continuous with respect to the Lebesgue measure~$\Lambda$ on~$\mathbb{X}$ with density\footnote{This follows from adapting the arguments of~\cite[Problem~4.25, Section~5.7.B, in particular equation~(5.7.26)]{KaSh12} for a bounded domain and from the proof of Lemma~\ref{lem:kernel_bounds} in Appendix~\ref{app:kernel_prop}.}~$k_t$, i.e.,
\begin{equation}\label{eq:kernel}
\begin{aligned}
\pi_t^{\#}\mu^{\gamma,x}(dy) =: \mu_t^{\gamma,x}(dy) =: &\ k_t(x,y,\gamma)\, dy, \\
&\ k \, : (t,x,y,\gamma) \mapsto k_t(x,y,\gamma) \; , \;  [0,T] \times \mathbb{X} \times \mathbb{X} \times V \rightarrow \R\,.
\end{aligned}
\end{equation}
The function $k$ is called the \emph{transition kernel}. Consider the following lemma, which is proven in Appendix~\ref{app:kernel_prop}.
\begin{lem} \label{lem:kernel_bounds}
For any sufficiently small~$\varepsilon>0$ there exists a constant~$\overline{K} < \infty$ such that for every~$\gamma$ that satisfies~$\|\gamma\|_V \le \varepsilon$, and for every~$x,y \in \mathbb{X}$,
\begin{equation} \label{eq:kernel_bounds}
k_t(x,y,\gamma) \le \overline{K}\,.
\end{equation}
In particular,~$k_t(\cdot,\cdot,\gamma) \in L^2(\Lambda \otimes \Lambda)$, and $f\in L^2(\Lambda) \Longrightarrow f\in L^2(\mu_t^{\gamma,x})$.
\end{lem}
The Perron--Frobenius operator~$\mathcal{P}_t(\gamma)$ and Koopman operator~$\mathcal{U}_t(\gamma)$ can be expressed with the transition kernel according to
\begin{equation}\label{eq:operators}
	\mathcal{P}_t(\gamma) f(y) = \int_{\mathbb{X}} k_t(x,y,\gamma) f(x) \, dx \quad \text{and} \quad \mathcal{U}_t(\gamma) g(x) = \int_{\mathbb{X}} k_t(x,y,\gamma) g(y) \, dy
\end{equation}
for an initial density~$f \in L^2(\Lambda)$ and an observable~$g \in L^2(\Lambda)$ at initial time $t=0$. 
By~\eqref{eq:expectation_of_path_functional_as_function_of_gamma}, the Koopman operator has the useful representation
\begin{equation*}
	\mathcal{U}_t(\gamma) g(x) = \bb{E} \left[g(X_t^{\gamma}) \, | \, X_0^{\gamma} = x\right] = u_g^x(\gamma)\,,
\end{equation*}
where we abuse the notation~$u_g^x(\gamma)$ to denote~$u_{g\circ \pi_t}^x(\gamma)$ if~$g\in L^2(\Lambda)$.
Using Theorem~\ref{thm:frechet_expectation}, this implies ``point-wise'' differentiablility of~$\gamma \mapsto \mathcal{U}_t(\gamma)g(x)$ for fixed~$x$ and~$g$. In the following we will lift this property to the operator level. Using the transition kernel, duality arguments, and uniformity of certain bounds in~$x\in\mathbb{X}$, we will be able to extend our result to the Perron--Frobenius operator and Koopman operator, where the differentiable dependence is with respect to the \emph{operator norm}. To this end, we will first dispose of~$g$ in Lemma~\ref{lem:functional}, then get rid of~$x$ in Lemma~\ref{lem:kernel}, and bring all this together in Theorem~\ref{thm:PFO_dif}. Below, we write $X^\ast$ to denote the topological dual of a given vector space $X$.
\begin{lem}\label{lem:functional}
	The operator $\gamma \mapsto \mathcal{U}_t(\gamma) (\cdot ) (x)$, considered as a mapping from $V$ to~$L^{2}( \Lambda )^{\ast}$, is Fr\'echet differentiable with respect to~$\gamma$.
\end{lem}
\begin{proof}
	In order to lift our point-wise result Theorem~\ref{thm:frechet_expectation} to the operator level, we use some quantitative bounds derived in the proof of that theorem. By Lemma~\ref{lem:kernel_bounds}, if~$g \in L^2( \Lambda )$ then~$g \in L^2\big(\mu_t^{0,x}\big)$ holds, and we have
	\begin{equation} \label{eq:marginal_derivative}
		D_{\beta} u_g^x \big\vert_{\beta = 0} (\gamma) = \bb{E}^{0,x} \left[(g\circ\pi_t) M^{\gamma}_t \right]
	\end{equation}
	for the derivative of $u_g^x$ with respect to $\gamma$ at $\gamma = 0$. The equations~\eqref{eq:lim_frechet1} and~\eqref{eq:lim_frechet2} imply the existence of the residual 
	\begin{equation} \label{eq:derivative_residual}
		r_g^x (\gamma) := u_g^x(\gamma) - u_g^x(0) - D_{\beta} u_g^x\big\vert_{\beta= 0} (\gamma) 
	\end{equation}
	with the property
	\begin{equation*}
		\lim_{\| \gamma \|_{\diamondsuit} \rightarrow 0} \dfrac{| r_g^x(\gamma) |}{\|\gamma \|_{\diamondsuit}} = 0 \,,
	\end{equation*}
	where $\diamondsuit = V$ or~$\diamondsuit = L^{\infty}\left( \mu^{0,x} \right)$. Now, by substituting equation~\eqref{eq:lim_frechet1} and~\eqref{eq:marginal_norm} in \eqref{eq:derivative_residual}, and by applying~\eqref{eq:kernel_bounds} and~\eqref{eq:bound}, it follows that there exists a constant~$C$ independent of~$x,g$ and~$\gamma$, and a positive function~$q (\gamma)$ independent of~$x$ and~$g$ such that
	\begin{equation} \label{eq:remainder_bound}
		| r_g^x(\gamma) | \leq C \| g\|_{L^2 ( \mu_t^{0,x}) } q(\gamma) \leq C\overline{K}^{1/2}\, \| g\|_{L^2 ( \Lambda )} q(\gamma) \,,
	\end{equation}
	where~$\frac{q(\gamma)}{\|\gamma\|_{\diamondsuit}}$ vanishes for~$\| \gamma \|_{\diamondsuit} \rightarrow 0$. Thus, if we consider~$u_g^x(\gamma)$ and its derivative~\eqref{eq:marginal_derivative} as linear functionals on~$g \in L^2(\Lambda)$, i.e.,
	\begin{align*}
		u_{(\cdot)}^x(\gamma): g &\mapsto \mathcal{U}_t(\gamma) ( g ) (x) =  \langle g, k(x,\cdot, \gamma)\rangle_{L^2(\Lambda)}\,, \\ 
		D_{\beta} u_{(\cdot)}^x \big\vert_{\beta = 0} (\gamma): g &\mapsto \bb{E}^{0,x} \left[(g\circ\pi_t) M^{\gamma}_t \right] \,
	\end{align*}
	then dividing~\eqref{eq:derivative_residual} by~$\| g \|_{L^2( \Lambda )}$, taking the supremum over all~$g\neq 0$, and using~\eqref{eq:remainder_bound}, we obtain
	\begin{equation} \label{eq:dual_bound1}
		0 \leq \dfrac{ \left\| u_{(\cdot)}^{x}(\gamma) - u_{(\cdot)}^x(0) - D_{\beta} u^x_{(\cdot)}|_{\beta=0} (\gamma) \right\|_{ L^2 (\Lambda)^{\ast} } }{\| \gamma \|_{\diamondsuit}} \leq C\overline{K}^{1/2}\, \frac{q(\gamma)}{\|\gamma\|_{\diamondsuit}}.
	\end{equation}
	Using that~$\frac{|q(\gamma)|}{\|\gamma\|_{\diamondsuit}} \to 0$ as~$\|\gamma\|_{\diamondsuit}\to 0$ proves the claim.
\end{proof}

\begin{remark} \label{rem:functional}
	By equations~\eqref{eq:ineq01} and~\eqref{eq:frechet-estimate} we can explicitly use
	\begin{equation}
		q(\gamma)^{2} = \frac{1}{2} (\lambda_\sigma\sqrt{T})^4 \|\gamma \|_{\diamondsuit}^4 + 2C \sum_{\ell=4}^{\infty} \sum_{\substack{j,k\geq 2\\ j+k=\ell}} \frac{1}{j!k!} (2\lambda_\sigma\sqrt{T})^{\ell} \| \gamma \|_{\diamondsuit}^{\ell} \; . \nonumber
	\end{equation}
	Further, we get for the derivative 
	\begin{equation}
		D_{\beta} u^x_{(\cdot)}\big\vert_{\beta=0} (\gamma) = \bb{E}^{0,x}\left[ \cdot M_t^{\gamma} \right] \, . \nonumber
	\end{equation}
\end{remark}

By using that the bounds \eqref{eq:remainder_bound} and \eqref{eq:dual_bound1} are uniform with respect to~$x$, we remove the ~$x$-dependence of the previous lemma in the following result.
\begin{lem}\label{lem:kernel}
	The kernel~$k$, or more specifically, the mapping $\gamma \mapsto k_t(\cdot , \cdot ,\gamma)$ from $V$ to~$L^2(\Lambda \otimes \Lambda)$ is Fr\'echet differentiable with respect to $\gamma$.
\end{lem}

\begin{proof}
	Considering Lemma~\ref{lem:functional}, we still need to get rid of the evaluation in~$x$. Fortunately, all relevant bounds are uniform in~$x$. First, we use the Riesz isomorphism $\mathcal{R}$ to identify~$L^2(\Lambda)^{\ast}$ with~$L^2(\Lambda)$, which leads to
	\begin{equation}
		\mathcal{R} \big( \mathcal{U}_t(\gamma) (\cdot ) (x) \big) = \mathcal{R} \big( u_{(\cdot)}^x(\gamma) \big) = k_t(x,\cdot ,\gamma), \nonumber
	\end{equation}
and further guarantees the existence of the following operator in $\gamma$,
	\begin{equation}\label{eq:kernel_derivative}
		D_{\beta} k_t(x,\cdot,\beta)|_{\beta=0}(\gamma) := \mathcal{R} \big( D_{\beta} u^x_{(\cdot)} \big\vert_{\beta=0} (\gamma) \big). \nonumber
	\end{equation}
	Note that the operator is a bounded linear operator. Since~$\mathcal{R}$ is a linear isomorphism, Lemma~\ref{lem:functional} and in particular~\eqref{eq:dual_bound1} provide
	\begin{equation}
		\big\| k_t(x, \cdot, \gamma) - k_t(x, \cdot, 0) - D_{\beta} k_t(x,\cdot ,\beta)\big\vert_{\beta=0} (\gamma) \big\|_{L^2(\Lambda)} \leq C\overline{K}^{1/2}\, | q(\gamma) |. \nonumber
	\end{equation}
	Finally, differentiability of~$\gamma \mapsto k_t(\cdot, \cdot, \gamma) \, , \, V  \rightarrow L^2(\Lambda \otimes \Lambda)$ is proven by observing that
	\begin{align}
		& \big\| k_t(\cdot , \cdot , \gamma) - k_t (\cdot , \cdot , 0) - D_{\beta} k_t(\cdot , \cdot , \beta)\big\vert_{\beta=0} (\gamma) \big\|_{L^2(\Lambda \otimes \Lambda)} \nonumber \\
		& \qquad = \left(\int_{\mathbb{X}} \big\| k_t(x , \cdot , \gamma) - k_t (x , \cdot , 0) - D_{\beta} k_t(x , \cdot , \beta)\big\vert_{\beta=0} (\gamma) \big\|_{L^2(\Lambda)}^2 \, dx \right)^{1/2} \nonumber \\
		& \qquad \leq \Lambda(\mathbb{X})^{1/2}\, C\overline{K}^{1/2}\, |q(\gamma)|  \nonumber
	\end{align}
	and noting that~$\frac{|q(\gamma)|}{\|\gamma\|_{\diamondsuit}} \to 0$ as~$\|\gamma\|_{\diamondsuit}\to 0$.
\end{proof}

We recall that the Perron--Frobenius operator and the kernel are related by~$\mathcal{P}_t(\gamma) f(y) = \int_{\mathbb{X}} k_t(x,y,\gamma) f(x) \, dx$. The differentiability of~$\mathcal{P}_t(\gamma) $ is now a simple consequence of Lemma~\ref{lem:kernel}. Given two vector spaces $U_1$ and $U_2$, we write $\bm{L}(U_1,U_2)$ to denote the space of bounded linear operators from $U_1$ to $U_2$.

\begin{theorem}\label{thm:PFO_dif}
The Perron--Frobenius operator $\mathcal{P}_t:V\to \bm{L}(L^2(\Lambda),L^2(\Lambda))$ that maps $\gamma\in V$ to $\mathcal{P}_t(\gamma)$ is Fr\'{e}chet differentiable with respect to~$\gamma$.
\end{theorem}
 
\begin{proof}
	Obviously, the mapping~$k_t(\cdot,\cdot,\gamma)\mapsto \mathcal{P}_t(\gamma)$ is linear. If this mapping is bounded, then it is arbitrarily often continuously differentiable. We consider $\mathcal{P}_t(\gamma) \in \bm{L}(L^2(\Lambda), L^2(\Lambda))$ with the induced norm. By the Cauchy--Schwarz inequality we obtain
	\begin{align*}
		\sup_{f\neq 0} \| f\|_{L^2(\Lambda)}^{-1} \| \mathcal{P}_t(\gamma) f\|_{L^2(\Lambda)} &= \sup_{f\neq 0} \| f \|_{L^2(\Lambda)}^{-1} \Big \|\! \int_{\mathbb{X}} k_t(x,\cdot,\gamma ) f(x) \, dx \Big \|_{L^2(\Lambda)} \\
			&= \sup_{f\neq 0} \| f \|_{L^2(\Lambda)}^{-1} \left(\int_{\mathbb{X}} \left| \int_{\mathbb{X}} k_t(x,y,\gamma) f(x) \, dx \right|^2 \, dy \right)^{\frac{1}{2}}\\
			& \leq \sup_{f\neq 0} \| f \|_{L^2(\Lambda)}^{-1} \left( \int_{\mathbb{X}} \| k_t(\cdot , y ,\gamma)\|^2_{L^2(\Lambda)} \| f \|^2_{L^2(\Lambda)} \, dy  \right)^{\frac{1}{2}}\\
			&=  \left(\iint_{\mathbb{X}\times \mathbb{X}} |k_t(x,y,\gamma)|^2 \, dy \, dx \right)^{\frac{1}{2}}= \| k_t(\cdot , \cdot,\gamma) \|_{L^2(\Lambda \otimes \Lambda)}
	\end{align*}
	and conclude
	\begin{equation*}
		\| \mathcal{P}_t(\gamma) \|_{ \bm{L} ( L^2(\Lambda), L^2(\Lambda) ) } \leq \| k_t(\cdot ,\cdot ,\gamma ) \|_{L^2(\Lambda \otimes \Lambda)}.
	\end{equation*}
	Now, Lemma~\ref{lem:kernel} and the chain rule for~$ \gamma \mapsto k_t(\cdot,\cdot,\gamma) \mapsto \mathcal{P}_t(\gamma)$ implies Fr\'echet differentiability of~$\mathcal{P}_t(\gamma)$ with respect to~$\gamma$.
\end{proof}

\begin{remark}
	We could prove the differentiability of the Koopman operator~$\mathcal{U}_t(\gamma)$ analogously to Theorem~\ref{thm:PFO_dif}. Alternatively, we can use the reasoning from Lemma~\ref{lem:chain_rule} and Remark~\ref{rem:chain_rule} to deduce the differentiability of the Koopman operator with respect to~$\gamma$.
\end{remark}


\section{Differentiability of isolated eigenvalues}
\label{sec:EVD}

With a classical result stated in e.g.,~\cite[p.2]{Kloeckner2017} and originally derived in~\cite{Rosenbloom1955} using the implicit function theorem, we can state the following result on smooth drift dependence of the eigenvalues and eigenvectors of the transfer operators. Note that all non-zero eigenvalues are automatically isolated, as the transfer operator is compact; see \cite[Lemma~30]{FrKo17} and the remark following the proof (note that this result assumes stronger regularity conditions on the data). Recall that the space $V$ from Section \ref{sec:prereq} equipped with the norm $\Vrt{\cdot}_V$ defined in \eqref{eq:norm_on_class_V} is a Banach space, cf.~Lemma~\ref{app:C-vec} in the appendix.

\begin{theorem} \label{thm:spectrum_dif}
Let us assume that $\lambda^0 $ is a simple and isolated eigenvalue with eigenvector~$f^0$ of the unperturbed linear and bounded operator $\mathcal{P}_t(0)$ that belongs to $\bm{L} \big( L^2(\Lambda),L^2(\Lambda) \big)$. Then, there exists a neighborhood~$U \subset V$ of the constant function~$0$ such that for all~$\gamma \in U$ the operators~$\mathcal{P}_t(\gamma)$ have an isolated eigenvalue~$\lambda^{\gamma}$ close to~$\lambda^0$. Further, the mappings~$\gamma \mapsto \lambda^{\gamma}$ and~$\gamma \mapsto f^{\gamma}$ that send the function $\gamma$ to its corresponding eigenvalue and eigenvector respectively are continously Fr\'echet differentiable. In the case of the eigenvector, which is only unique up to scaling, the scaling can be chosen such that the given map is differentiable.
\end{theorem}
As with Theorem \ref{thm:PFO_dif}, one can show the analogous result for the Koopman operator.
\begin{proof}
	By the Fr\'echet differentiability of~$\gamma \mapsto \mathcal{P}_t(\gamma)$ from Theorem~\ref{thm:PFO_dif} we have
	\begin{equation}
		\mathcal{P}_t(\gamma) = \mathcal{P}_t(0) + \mathcal{Q}_t(\gamma),\quad \mathcal{Q}_t(\gamma) = \mathcal{O}\left( \| \gamma \|_V \right). \nonumber
	\end{equation}
Now, we can consider~$\mathcal{P}_t(\gamma)$  as an additive perturbation of~$\mathcal{P}_t(0)$ in operator space~$\bm{L}$. Using the idea of \cite{Rosenbloom1955}, nicely stated in~\cite[p.1]{Kloeckner2017}, we can deduce the existence of $\bm{U} \subset \bm{L}$, a neighborhood of $\mathcal{P}_t(0)$, and mappings $m_1 \, : \, \mathcal{P}_t(\gamma) \mapsto \lambda^{\gamma}$ and $m_2\, : \, \mathcal{P}_t(\gamma) \mapsto f^{\gamma}$ such that $m_1$ and $m_2$ are analytical on $\bm{U}$.
The proof in \cite{Kloeckner2017} utilizes the implicit function theorem to conclude this. By Theorem~\ref{thm:PFO_dif}, $\ell \, : \, \gamma \mapsto \mathcal{P}_t(\gamma)$ is a continuously differentiable mapping. Thus, there exists a neighborhood $U\subset V$ of $0$ such that the mappings $n_1= m_1 \circ \ell \, : U \rightarrow \R \, , \gamma \mapsto \lambda^{\gamma}$ and $n_2 = m_2 \circ \ell \, : \, U \rightarrow L^2(\Lambda) \, , \, \gamma \mapsto f^{\gamma}$ are continuously Fr\'echet differentiable.
\end{proof}

\begin{remark}[Changing reference measures] \label{rem:changeing_reference_measure}
For non-stationary dynamics it is natural to consider $\mathcal{P}_t \, : \, L^p(\mu_0) \rightarrow L^p(\mu_t)$, cf.~\cite{Fro13,Den17}, where~$\mu_0$ is some given initial distribution and~$\mu_t$ is the push-forward of~$\mu_0$ by the dynamics.\footnote{The reason for this is that the transfer operator is a well-defined contraction between these spaces for any~$1\le p\le\infty$, even for purely deterministic dynamics; see~\cite[Theorem~5, pp.~21]{Den17}.} The $\gamma$-dependence of~$\mu_t = \mu_t^{\gamma}$, however, poses the problem that it is not trivial to compare the different $\mathcal{P}_t({\gamma}) f \in L^p (\mu_t^{\gamma})$ with one another. Thus, the natural approach here is to work on the ``common space''~$L^p(\Lambda)$. Results that consider the case where the perturbed operators map to different spaces~\cite{Kol06, ZhPa07, MNP13} are more involved, and beyond the scope of the present paper.
\end{remark}

Finally, let us discuss two applications that illustrate the significance of the preceding results in the context of \emph{dynamical systems}.
\paragraph{Singular values and functions.}
As both $\mathcal{P}_t(\gamma)$ and $\mathcal{U}_t(\gamma)$ are differen\-tiable with respect to~$\gamma$ by Theorem~\ref{thm:PFO_dif}, it follows that their concatenations $\mathcal{P}_t(\gamma)\mathcal{U}_t(\gamma)$ and $\mathcal{U}_t(\gamma)\mathcal{P}_t(\gamma)$ are also differentiable with respect to~$\gamma$. However, since $\mathcal{P}_t(\gamma)$ and~$\mathcal{U}_t(\gamma)$ are adjoints (cf.~\eqref{eq:operators} or~\cite{LaMa94}), we have the differentiability of~$\mathcal{P}_t(\gamma)^*\mathcal{P}_t(\gamma)$ and~$\mathcal{P}_t(\gamma)\mathcal{P}_t(\gamma)^*$, thus of the singular values, and thus of the right and left singular vectors of~$\mathcal{P}_t(\gamma)$ and~$\mathcal{U}_t(\gamma)$. These are of particular interest for non-autonomous systems, as studies by Froyland et al.~\cite{FrSaMo10, Fro13} have shown that they can be connected to finite-time persistent dynamical structures called \emph{coherent sets}. It should be noted that these studies consider the operator~$\mathcal{P}_t$ with respect to changing reference measures, as discussed above in Remark~\ref{rem:changeing_reference_measure}. Nevertheless, our results apply to the large class of systems with divergence-free velocity fields (modeling incompressible flows) and with diffusion coefficients independent of the spatial variable~$x$. For such systems the Lebesgue measure is an invariant reference measure.

\paragraph{Periodically forced systems.}
For periodically forced non-autonomous systems, i.e., where~$b(t+T,\cdot) = b(t,\cdot)$, $\gamma(t+T,\cdot) = \gamma(t,\cdot)$, and~$\sigma(t+T,\cdot) = \sigma(t,\cdot)$ for some~$T>0$, Theorem~\ref{thm:spectrum_dif} shows smooth drift-dependence of the long-term behavior, e.g., of ergodic averages of the type
\[
\bar{g}^{\gamma}(x) := \lim_{t\to\infty} \frac{1}{t}\int_0^t g(\tau\text{ mod } T,X^{\gamma}_{\tau})\,d\tau,\quad X^{\gamma}_0=x\,,
\]
where~$g:[0,T)\times\mathbb{X}\to \Reals$.
To see this, let~$\mathcal{P}_{t_0,t_1}(\gamma)$ denote the Perron--Frobenius operator for the SDE~\eqref{eq:perturbed_sde} running from time~$t_0$ to~$t_1$, and let~$\{f^{\gamma}_t\}_{t\in [0,T)}$ be the \emph{stationary family of densities}, i.e.,~$\mathcal{P}_{s,t}(\gamma)f^{\gamma}_s = f^{\gamma}_{t\text{ mod }T}$. In particular,~$\mathcal{P}_{s,s+T}(\gamma) f^{\gamma}_s = f^{\gamma}_s$, and the eigenfunctions are differentiable with respect to~$\gamma$. As the process~$X^{\gamma}$ is clearly irreducible due to non-degenerate noise, the dominant eigenvalue~$1$ of~$\mathcal{P}_{s,s+T}(\gamma)$ is simple. By Birkhoff's individual ergodic theorem
\[
\bar{g}^{\gamma}(x) = \frac{1}{T}\int_0^T\!\!\! \int_{\mathbb{X}} g(\tau,y)f^{\gamma}_{\tau}(y)\,dy\,d\tau
\]
$\bb{P}^x$-almost surely for~$\Lambda$-a.e.~$x$, which is continuously differentiable in~$\gamma$.

\section*{Acknowledgments}

PK is supported by the Deutsche Forschungsgemeinschaft (DFG) through the CRC 1114 ``Scaling Cascades in Complex Systems'', project A01, and by the Einstein Foundation Berlin (Einstein Center ECMath). HCL is supported by the Freie Universit\"{a}t Berlin within the Excellence Initiative of the DFG. MP is supported by the DFG through the Priority Programme SPP 1881 ``Turbulent Superstructures''.


\appendix
\section{Appendix}
\subsection{Additional Material}

\begin{lem}\label{lem:chain_rule}
	Let $U,V,W$ and $Z$ be real or complex Banach spaces with their respective norms. Consider a family of operators $(A(z))_{z\in Z} \subset L(U,V)$, a linear and bounded operator $j \, : L(U,V) \rightarrow W$ and the mapped family~$(j(A(z)))_{z\in Z}$. Now, if $z \mapsto A(z)$ is Fr\'echet differentiable then $z \mapsto j(A(z)) =: B(z)$ is Fr\'echet as well, and it holds for the derivative
	\begin{equation*}
	D_{z} B|_{z=z_0} = j \circ (D_{z} A|_{z=z_0}) \; .
	\end{equation*}
\end{lem}
\begin{remark}\label{rem:chain_rule}
	More important than the simple and short proof given below are some special cases we are interested in. For $W = L(V^{\ast}, U^{\ast})$ and the linear isometry $j\, : \, A \mapsto A^{\ast}$ the above lemma implies the differentiability of the adjoint family $(A^{\ast}(z))_{z \in Z}$. If in this special case $U,V$ are also reflexive spaces, then the reverse direction of the lemma holds as well.
\end{remark}
\begin{proof}
	The claims simply follow from the chain rule for Fr\'echet derivatives and the rule for differentiating linear operators.
\end{proof}

The next result considers existence and uniqueness of strong solutions. It is adapted from~\cite[Section 4.4.2, Corollary]{liptsershiryaevpartI}.
\begin{theorem}\label{thm:theorem4_6_cor}
	Consider the stochastic differential equation
	\begin{equation}\label{eq:4_134}
		dX_t=b(t,X_t)dt+\sigma(t,X_t)dW_t,\quad X_0=\eta,
	\end{equation}
	where the functions $b:[0,T]\times \Reals^d\to\Reals^d$, $\sigma:[0,T]\times\Reals^d\to \Reals^{d\times d}$ satisfy the assumptions (A1) \eqref{eq:A1-Lip} and (A2) \eqref{eq:A2-Growth}. If $\Prob(\sum_{i=1}^{n}\abs{\eta_i}<\infty)=1$, then \eqref{eq:4_134} has a unique strong solution.
\end{theorem}

\begin{lem}\label{lem:A4}
	The assumptions (A1), (A2) and (A3)  imply the property (A4).
\end{lem}
\begin{proof}
Observe that, for arbitrary $(s,y)\in[0,T]\times\bb{R}^d$,
\begin{align*}
 \left(\abs{\sigma^{-1}_s(b_s+\gamma_s)}^2+\abs{\sigma^{-1}_sb_s}^2\right)(y)&\leq \lambda_\sigma^{2}\abs{b_s+\gamma_s}^2(y)+\lambda_\sigma^{2}\abs{b_s}^2(y)
 \\
 &\leq 3\lambda_\sigma^{2}\left(\abs{b(s,y)}^2+\abs{\gamma(s,y)}^2\right)\\
 &\leq 6\lambda_\sigma^{2}L^2(1+\abs{y})^2\\
 &\leq 12\lambda_\sigma^{2}L^2(1+\abs{y}^2),
\end{align*}
where the first inequality follows from \eqref{eq:uniform_ellipticity}, the second and fourth inequalities from the inequality $(a+b)^2\leq 2(a^2+b^2)$, and the third inequality from \eqref{eq:A2-Growth}. Therefore, in order to show that \eqref{eq:P_bound} holds, it suffices to show that
\begin{equation}\label{eq:no_explosion_consequence}
 \bb{P}^x\left(\int_0^T\abs{X^0_s}^2ds<\infty\right)=1.
\end{equation}
By \cite[Corollary after Theorem~2.2, pp.~199]{narita}, if $b$ and $\sigma$ are globally Lipschitz, if there exists some continuous function $\alpha:[0,\infty]\to[0,\infty)$ and a monotone, increasing, concave $\beta:[0,\infty]\to[0,\infty)$ such that
\begin{equation}\label{eq:no_explosion_criterion_2}
 \int_0^\infty \frac{1}{1+\beta(u)}du=\infty,
\end{equation}
and if
\begin{equation}\label{eq:no_explosion_criterion_1}
 2 y^Tb(t,y) +\abs{\sigma(t,y)}^2\leq \alpha(t)\beta(\abs{y}^2),\quad\forall (t,y)\in[0,T]\times \bb{R}^d,
\end{equation}
then the explosion time of $X^0$ is infinite, $\bb{P}^x$-almost surely (see the discussion after \cite[Equation (1.4)]{narita} for the definition of the explosion time). If the explosion time of $X^0$ is infinite $\bb{P}^x$-almost surely, then
\begin{equation}\label{eq:no_explosion}
 \bb{P}^x\left( \abs{X^0_s}<\infty \, ,\forall s\in [0,T]\right)=1.
\end{equation}
Of course, if $X^0_s$ takes values in a bounded set, then \eqref{eq:no_explosion} holds immediately. Since $X^0$ (and thus $\abs{X^0}$) has continuous paths $\bb{P}^x$-almost surely, it follows that \eqref{eq:no_explosion} implies \eqref{eq:no_explosion_consequence}, by the extreme value theorem. Now it remains to show that \eqref{eq:no_explosion_criterion_1} and \eqref{eq:no_explosion_criterion_2} are satisfied, given that $b$ and $\sigma$ are globally Lipschitz functions on their domains. Using the Cauchy--Schwarz inequality, Young's inequality, and \eqref{eq:A2-Growth}, it follows that
\begin{align*}
  y^Tb(t,y) & \leq \abs{y}^2+\abs{b(t,y)}^2\leq \abs{y}^2+L^2(1+\abs{y})^2\leq C'(1+\abs{y}^2),
 \\
 \abs{\sigma(t,y)}^2&\leq C(1+\abs{y})^2\leq C'(1+\abs{y}^2),
\end{align*}
for some appropriately chosen $C'>0$. By combining the preceding inequalities, we can conclude that, for globally Lipschitz $b$ and $\sigma$, \eqref{eq:no_explosion_criterion_1} is satisfied for $\alpha(t):=2C'$, and for $\beta(u):=1+u$ for $u\in[0,\infty)$. Note that $\alpha$ is continuous, and that $\beta$ is linear, and hence monotone, increasing, and concave; in addition, $\beta$ satisfies \eqref{eq:no_explosion_criterion_2}.
\end{proof}

\begin{lem}\label{app:C-vec} \label{app:C-norm} \label{app:L-nvs}
	The tuple $(V,\| \cdot \|_V)$ forms a Banach space, i.e.,~$V:= \mathcal{C} \cap L^{\infty}(\mu^{0,x})$ is closed with respect to~$\| \cdot \|_{V}$.
\end{lem}
\begin{proof}
	We verify that $\Vrt{\cdot}_V$ indeed defines a norm on $V\coloneqq \mathcal{C}\cap L^\infty(\mu^{0,x})$. Note that the linear growth condition \eqref{eq:A2-Growth} is a weaker condition than the boundedness condition \eqref{eq:almost_surely_bounded_changes_of_drift}, and that Lipschitz continuity combines with \eqref{eq:almost_surely_bounded_changes_of_drift} to ensure that any $f\in \mathcal{C}\cap L^\infty(\mu^{0,x})$ is everywhere bounded. Thus, by the Lipschitz continuity and boundedness conditions, we have $\Vrt{f}_V=0$ if and only if $f$ is constant and equal to zero. Using the definition of $V$ and its norm, we observe that if $f,g\in V$ then $\Vrt{f+g}_V\leq \Vrt{f}_V+\Vrt{g}_V$, and also that $\Vrt{\alpha f}_V=\abs{\alpha}\Vrt{f}_V$ for all $\alpha\in\bb{R}$. Observe also that $V$ consists of all functions whose $\Vrt{\cdot}_V$-norm is finite. To show that $V$ is closed under $\Vrt{\cdot}_V$, let $(f_n)_{n\in\bb{N}}\subset V$ converge in $V$ to some function $f:[0,T]\times\bb{R}^d\to\bb{R}^d$. Since $f_n\in V$ and since $(\Vrt{f_n-f}_V)_{n\in\bb{N}}$ by definition converges to zero, we have
	\begin{equation*}
		\Vrt{f}_V\leq \Vrt{f_n-f}_V+\Vrt{f_n}_V,\quad \text{for all } n\in\bb{N}.
	\end{equation*}
	Both quantities on the right-hand side are finite for any $n\in\bb{N}$, so $f\in V$ and hence $f$ is Lipschitz continuous. Since $L^\infty(\mu^{0,x})$ contains discontinuous functions, it follows that $V$ is a normed vector space and a proper subset of $L^\infty(\mu^{0,x})$. The completeness can be proven using adaptations of the arguments above.
\end{proof}

\subsection{Proofs of lemmas~\ref{lem:Z_L2} and~\ref{lem:tech2}}
\label{app:proof-lem2}

Recall that we use $\langle M\rangle$ to denote the quadratic variation process of a continuous local martingale~$M$, as in~\eqref{eq:M_gamma_q}. The following result is given in \cite[Section 1.2, Theorem 1.5]{kazamaki}.
\begin{theorem}\label{thm:kazamakis_theorem}
	Let $M$ be a continuous local martingale. Let $1<p,p'<\infty$ be conjugate H\"{o}lder exponents, i.e., $p^{-1}+(p')^{-1}=1$. If
	\begin{equation}\label{eq:kazamakis_condition_for_Lq_boundedness_of_exp_MG}
		\sup\left\{\bb{E}\left[\exp\left(\frac{1}{2}\frac{\sqrt{p'}}{\sqrt{p'}-1}M_T\right)\right]\ \biggr\vert\ T\text{ a bounded stopping time }\right\}<\infty, \nonumber
	\end{equation}
	then the exponential martingale $(\exp(M_t-\tfrac{1}{2}\langle M\rangle_t))_{t\in[0,T]}$ is $L^p$-bounded. In particular, it holds that $\bb{E}\left[\exp(M_t-\tfrac{1}{2}\langle M\rangle_t)^p\right]$ is finite, for all $t\in [0,T]$.
\end{theorem}
Now we use Theorem \ref{thm:kazamakis_theorem} to prove Lemma \ref{lem:Z_L2}.
\begin{proof}[Proof of Lemma \ref{lem:Z_L2}]
	Recall that we abbreviate~$\Vrt{\cdot}_{L^\infty(\mu^{0,x})}$ as~$\Vrt{\cdot}_{L^\infty}$. Given property \eqref{eq:uniform_ellipticity},
	\begin{equation*}
		\abs{\sigma^{-1}\gamma(s,y)}\leq \lambda_\sigma^{2} \Vrt{\gamma}_{L^\infty}\quad \text{ for all } (s,y)\in[0,T]\times\Reals^d
	\end{equation*}
	holds, and thus by the definition of $\langle M^\gamma\rangle_t\left(X^0\right)$ we get
	\begin{equation}\label{eq:almost_sure_bound_on_qv}
		\left(\langle M^\gamma\rangle_t (X^0)\right)^r\leq  \left(\lambda_\sigma^{2} \Vrt{\gamma}^2_{L^\infty} t \right)^r\leq \left(\lambda_\sigma\sqrt{t}\Vrt{\gamma}_{L^\infty}\right)^{2r}\quad\text{for all } r\geq 1,
	\end{equation}
	uniformly in $x$. Then,~\eqref{eq:almost_sure_bound_on_qv} implies
	\begin{equation}\label{eq:novikov_condition}
 		\bb{E}^{0,x}\left[\exp\left(\alpha\langle M^\gamma\rangle_t \right)\right]\leq \exp\left(\alpha \lambda_\sigma^2 t\Vrt{\gamma}^2_{L^\infty}\right),\quad\text{for all } \alpha \geq 0.
	\end{equation}
	Using $Z^\gamma_t=\exp(M^\gamma_t-\tfrac{1}{2}\langle M^\gamma\rangle_t)$, it follows that~$\exp(\tfrac{1}{2}M^\gamma_t)=(Z^\gamma_t)^{1/2}\exp(\frac{1}{4}\langle M^\gamma\rangle_t)$. Since $Z^\gamma_t$ is the Radon--Nikodym derivative of two probability measures, it holds that $\bb{E}^{0,x}[Z^\gamma_t]=1$. Using this fact and the Cauchy--Schwarz inequality, we obtain
	\begin{equation*}
		\bb{E}^{0,x}\left[\exp\left(\frac{1}{2}M^\gamma_t\right)\right]\leq \bb{E}^x\left[\exp\left(\frac{1}{2}\langle M^\gamma\rangle_t\right)(X^0)\right]^{1/2}.
	\end{equation*}
	Now, let $1<p'<\infty$, replace $\gamma$ with $(\sqrt{p'}/(\sqrt{p'}-1))\gamma$, use the linearity of the map $\gamma\mapsto M^\gamma_t$, and use that $\langle  M^{\alpha\gamma}\rangle_t=\alpha^2\langle M^\gamma\rangle_t$ for every $\alpha\in\Reals$, to obtain that
	\begin{equation*}
		\bb{E}^x\left[\exp\left(\frac{\sqrt{p'}}{2(\sqrt{p'}-1)}M^\gamma_t\right)(X^0)\right]\leq \bb{E}^x\left[\exp\left(\frac{p'}{2(\sqrt{p'}-1)^2}\langle M^\gamma\rangle_t\right)(X^0)\right]^{1/2}.
	\end{equation*}
	It follows by choosing $p'=p=2$, \eqref{eq:novikov_condition}, and Theorem~\ref{thm:kazamakis_theorem} that~$Z^\gamma_t(X^0)\in L^2(\Prob^x)$. 
\end{proof}

\begin{proof}[Proof of Lemma~\ref{lem:tech2}]
	By the binomial theorem, $(\abs{M^\gamma_t}+\tfrac{1}{2}\langle M^\gamma\rangle_t)^n$ involves summing $\abs{M^\gamma_t}^m\langle M^\gamma\rangle_t^{\ell-m}$ for all $0\leq m\leq \ell$. Using Young's inequality, we obtain for the H\"{o}lder conjugate pair $(p,p')\in (1,\infty)^2$ that
	\begin{equation*}
		\abs{M^\gamma_t}^{\ell-m}\langle M^\gamma\rangle_t^{m}\leq \frac{1}{p}\abs{M^\gamma_t}^{p(\ell-m)}+\frac{1}{p'}\langle M^\gamma\rangle_t^{mp'}.
	\end{equation*}
	For $1\leq m\leq \ell-1$, setting $p=(\ell+m)/(\ell-m)$ and $p'=(\ell+m)/(2m)$ yields
	\begin{equation*}
		\abs{M^\gamma_t}^{\ell-m}\langle M^\gamma\rangle_t^{m}\leq \abs{M^\gamma_t}^{\ell+m}+\langle M^\gamma\rangle_t^{(\ell+m)/2}.
	\end{equation*}
	By Doob's inequality \cite[Equation (2.1)]{Peskir:1996}, we have 
	\begin{equation}\label{eq:doobs_inequality}
		\bb{E}^{0,x} \left[ \sup_{0\leq t\leq T}\abs{M^\gamma_t}^{r}\right]\leq \left(\frac{r}{r-1}\right)^{r} \bb{E}^{0,x} \left[ \langle M^\gamma\rangle_t^{r/2}\right],\quad\text{for all } r>1\,.
	\end{equation}
	Note that the map $r\mapsto (r/(r-1))^r$ decreases as $r$ increases. Therefore, given the preceding inequalities and \eqref{eq:almost_sure_bound_on_qv}, it follows that for $\ell \geq 2$ and $1\leq m\leq \ell-1$,
	\begin{equation}
		\bb{E}^{0,x} \left[ \abs{M^\gamma_t}^{\ell-m}\langle M^\gamma\rangle_t^{m}\right]\leq (C+1)\bb{E}^{0,x}\left[\langle M^\gamma\rangle_t^{(\ell+m)/2}\right]\leq (C+1)\left(\lambda_\sigma\sqrt{t}\Vrt{\gamma}_{L^\infty}\right)^{\ell+m} \label{eq:doobs_inequality_consequence_01}
	\end{equation}
	where $C=(4/3)^4$. Using the binomial theorem, \eqref{eq:doobs_inequality} and \eqref{eq:doobs_inequality_consequence_01}, and the fact that $C>1$, we obtain
	\begin{align}
		\bb{E}^{0,x}\left[\left(\abs{M^\gamma_t}+\frac{1}{2}\langle M^\gamma\rangle_t\right)^\ell\right]&=\sum^\ell_{m=0}\begin{pmatrix} \ell\\ m \end{pmatrix} \bb{E}^{0,x}\left[ \abs{M^\gamma_t}^{\ell-m}\langle M^\gamma\rangle_t^{m}\right]
 		\nonumber \\
		&=\bb{E}^{0,x}\left[\abs{M^\gamma_t}^\ell\right]+\sum^{\ell-1}_{m=1}\begin{pmatrix} \ell\\ m \end{pmatrix} \bb{E}^{0,x}\left[ \abs{M^\gamma_t}^{\ell-m}\langle M^\gamma\rangle_t^{m}\right]+\bb{E}^{0,x}\left[\langle M^\gamma\rangle_t^\ell\right]
		\nonumber\\
		&\leq C \bb{E}^{0,x}\left[\langle M^\gamma\rangle_t^{\ell/2}\right]+\sum^{\ell-1}_{m=1}\begin{pmatrix} \ell\\ m \end{pmatrix} \bb{E}^{0,x}\left[ \abs{M^\gamma_t}^{\ell-m}\langle M^\gamma\rangle_t^{m}\right]+\bb{E}^{0,x}\left[\langle M^\gamma\rangle_t^\ell\right]
		\nonumber\\
		&\leq C (\lambda_\sigma\sqrt{t}\Vrt{\gamma}_{L^\infty})^{\ell}+\sum^{\ell-1}_{m=1}\begin{pmatrix} \ell\\ m \end{pmatrix} (C+1)(\lambda_\sigma\sqrt{t} \Vrt{\gamma}_{L^\infty})^{\ell+m}+(\lambda_\sigma\sqrt{t}\Vrt{\gamma}_{L^\infty})^{2\ell}
		\nonumber\\
		&\leq C(\lambda_\sigma\sqrt{t}\Vrt{\gamma}_{L^\infty})^{\ell}\left(1+2\sum^{\ell-1}_{m=1}\begin{pmatrix} \ell\\ m \end{pmatrix} (\lambda_\sigma\sqrt{t} \Vrt{\gamma}_{L^\infty})^{m}+(\lambda_\sigma\sqrt{t}\Vrt{\gamma}_{L^\infty}^{\ell}\right)
		\nonumber\\
		&\leq 2C(\lambda_\sigma\sqrt{t}\Vrt{\gamma}_{L^\infty})^{\ell}\left(1+\lambda_\sigma\sqrt{t}\Vrt{\gamma}_{L^\infty}\right)^{\ell} \label{eq:doobs_inequality_consequence_02}
	\end{align}
	for $\ell\geq 2$. Observe that if $\Vrt{\gamma}_{L^\infty}\leq \lambda_\sigma\sqrt{t}^{-1}$, then $(1+\lambda_\sigma\sqrt{t}\Vrt{\gamma}_{L^\infty})^\ell\leq 2^\ell$, and from \eqref{eq:doobs_inequality_consequence_02} it follows that
	\begin{equation*}
		\bb{E}^{0,x}\left[\left(\abs{M^\gamma_t}+\frac{1}{2}\langle M^\gamma\rangle_t\right)^\ell\right]\leq C(\lambda_\sigma\sqrt{t}\Vrt{\gamma}_{L^\infty})^{\ell} 2^{\ell+1}.
	\end{equation*}
	We shall use the equation 
	\begin{equation*}
		\abs{\sum^\infty_{n=2}\frac{1}{n!}y^n}^2=\sum_{\ell=4}^\infty\sum_{\substack{j,k\geq 2\\ j+k=\ell}}\frac{1}{j!k!}y^\ell,
	\end{equation*}
	with $y=\abs{M^\gamma_t}+\tfrac{1}{2}\langle M^\gamma\rangle_t$ and $y=\sqrt{\lambda_\sigma\sqrt{t}\Vrt{\gamma}_{L^\infty}}$. By the dominated convergence theorem, we may interchange summation with expectation:
	\begin{align}
		\bb{E}^{0,x}\left[\left(\sum^\infty_{n=2}\frac{1}{n!}\left(\abs{M^\gamma_t}+\frac{1}{2}\langle M^\gamma\rangle_t\right)^n\right)^2\right]&=\bb{E}^{0,x}\left[\sum^\infty_{\ell=4}\sum_{\substack{j,k\geq 2\\ j+k=\ell}}\frac{1}{j!k!}\left(\abs{M^\gamma_t}+\frac{1}{2}\langle M^\gamma\rangle_t\right)^\ell\right]
		\nonumber\\
		&=\sum^\infty_{\ell=4}\sum_{\substack{j,k\geq 2\\ j+k=\ell}}\frac{1}{j!k!}\bb{E}^{0,x}\left[\left(\abs{M^\gamma_t}+\frac{1}{2}\langle M^\gamma\rangle_t\right)^\ell \right]
		\nonumber\\
		&\leq 2C \sum_{\ell=4}^\infty\sum_{\substack{j,k\geq 2\\ j+k=\ell}}\frac{1}{j!k!}\left(2\lambda_\sigma\sqrt{t}\Vrt{\gamma}_{L^\infty}\right)^{ \ell} \label{eq:bound01}  \nonumber\\
 		&=2C \big( \exp\left(\lambda_\sigma\sqrt{t}\Vrt{\gamma}_{L^\infty}\right)-1-2\lambda_\sigma\sqrt{t}\Vrt{\gamma}_{L^\infty} \big)^2.
		\nonumber
	\end{align}
\end{proof}

\section{Properties of the Transition Kernel}
\label{app:kernel_prop}

Via the connection between the SDE~\eqref{eq:original_sde} and the Fokker--Planck partial differential equation~\eqref{eq:FokkerPlanck} below, one can investigate properties of the transition kernel both by stochastic calculus and the theory of partial differential equations (PDEs). This has been done since the 1950's in great detail and generality, however, the particular case of kernel (or fundamental solutions) estimates for a time-dependent SDE (or parabolic PDE) with reflecting (or homogeneous Neumann) boundary conditions could not be found by the authors. There exist regularity results for the kernel for general SDEs~\cite[Section 5.7.B]{KaSh12} and quantitative upper and lower bounds for reversible SDEs~\cite{FaSt89} on~$\Reals^d$, further there are heat kernel (i.e., pure diffusion without drift) estimates for smooth manifolds with Neumann boundary conditions~\cite[Theorem~3.8]{Wang97}. It is unclear to the authors to which extent the very general results of Ivasi{\v s}en~\cite{Iva82a, Iva82b} apply in our case.

Thus, we shall derive properties of the transition kernel associated with the SDE~\eqref{eq:perturbed_sde} on a bounded domain~$\mathbb{X}$ with sufficiently smooth boundary and reflecting boundary conditions from estimates stemming from the theory of parabolic PDEs. In particular, we will be following the steps of Friedman~\cite{Fri08} and sharpen some of his estimates.

The general approach is the following. Associated with the SDE~\eqref{eq:original_sde} (and analogously, with~\eqref{eq:perturbed_sde}) is the Fokker--Planck (or forward Kolmogorov) equation~\cite[Section~11]{LaMa94}
\begin{equation} \label{eq:FokkerPlanck}
\frac{\partial f}{\partial t} =  \frac12 \sum_{i,j=1}^d \frac{\partial^2}{\partial x_i \partial x_j} \left( (\sigma\sigma^{\top})_{ij} f\right) - \sum_{i=1}^d \frac{\partial}{\partial x_i}(b_i f)
\end{equation}
for~$f = f(x,t)$, which describes the density evolution with respect to~\eqref{eq:original_sde}. Thus,
\[
f(x,t) = \int_{\mathbb{X}} k_t(\xi,x,\gamma)\vert_{\gamma=0}f(\xi,0)\,d\xi
\]
solves~\eqref{eq:FokkerPlanck}.
Reflecting boundary conditions for the SDE translate into homogeneous Neumann boundary conditions for~\eqref{eq:FokkerPlanck}, cf.~\cite[Theorem~3.1.1]{Pil14}. If this initial-boundary value problem admits the \emph{fundamental solution}~$\Gamma_N = \Gamma_N(x,t;\xi, \tau)$, i.e.,
\[
f(x,t) = \int_{\mathbb{X}} \Gamma_N(x,t;\xi,0) f_0(\xi)\,d\xi\,,
\]
where~$f_0 = f(\cdot,0)$ is the initial condition, then uniqueness of the solutions (see~\cite[Theorem~5.3.2]{Fri08} or~\cite[Section 6.12, Theorem 6.6]{Tan96} for the~$L^p$ case) gives~$ \Gamma_N(x,t;\xi,0) = k_t(\xi,x,\gamma)\vert_{\gamma=0}$, and bounding~$\Gamma_N$ shows Lemma~\ref{lem:kernel_bounds}.

When referring to equation~($k.\ell.m$), we mean equation~($\ell.m$) in~\cite[Chapter~$k$]{Fri08}. Further, we will partially adapt Friedman's notation, thus we write~\eqref{eq:FokkerPlanck} (and its analogue for the $\gamma$-perturbed SDE~\eqref{eq:perturbed_sde}) as~$\mathcal{L}f = 0$, where~$\mathcal{L}$ is a uniformly parabolic (with coercivity constant~$\lambda_\sigma^{-2}$ by~\eqref{eq:uniform_ellipticity}) differential operator with uniformly bounded, H\"older continuous (with some exponent~$\alpha>0$) coefficients and $\sigma$ being a temporal variable for now. Note that the assumptions at the beginning of Section~\ref{sec:transfer_ops} imply those in~\cite[sections~1.1 and 5.1]{Fri08}. It is also important to note that the bound~$\|\gamma\|_V \le \varepsilon$ implies that the different constants denoted by `$\const$' in the following all depend on~$\varepsilon$, and are hence uniform in~$\gamma$. The constants `$\const$' will always be independent of the spatial and temporal variables~$x,y,\xi, t,\sigma,\tau$.

To establish the desired bound on the fundamental solution, we will need an analogue of~\cite[Lemma~1.4.3]{Fri08}, with respect to integration on the boundary~$\partial\mathbb{X}$.
\begin{lem} \label{lem:Fri08_Lem1.4.3_analogue}
For~$-\infty < \kappa_1, \kappa_2 < \frac{d-1}{2}+1$, $h>0$, and 
\begin{align*}
|f(x,t;\xi,\tau)| &\le \const (t-\tau)^{-\kappa_1} \exp\left[ -\frac{h | x-\xi |^2}{4(t-\tau)} \right]\,, \\
|g(x,t;\xi,\tau)| &\le \const (t-\tau)^{-\kappa_2} \exp\left[ -\frac{h | x-\xi |^2}{4(t-\tau)} \right]\,,
\end{align*}
one has
\begin{equation} \label{eq:convolution_bound}
\begin{aligned}
\int_{\sigma}^t \int_{\partial \mathbb{X}} | f(x,t; \xi,\tau) g(\xi,\tau; y,\sigma) | \,dS_{\xi}d\tau &\le \const \left(\frac{4\pi}{h}\right)^{\frac{d-1}{2}} B\left( \frac{d-1}{2}-\kappa_1+1, \frac{d-1}{2}-\kappa_2+1\right)\cdot \\
&\phantom{\le} \quad \cdot (t-\sigma)^{d/2+1/2-\kappa_1-\kappa_2}\, \exp\left[ -\frac{h | x-y |^2}{4(t-\sigma)} \right],
\end{aligned}
\end{equation}
where~$S_{\xi}$ denotes the surface measure on~$\partial\mathbb{X}$ with respect to~$\xi$, and~$B(\cdot, \cdot)$ denotes the beta function.
\end{lem}
\begin{proof}[Proof of Lemma~\ref{lem:Fri08_Lem1.4.3_analogue}]
As~$\mathbb{X}$ is bounded, and~$\partial\mathbb{X}$ is of class~$C^{1+\delta}$, the set~$\partial\mathbb{X}$ can be decomposed into a finite number of patches such that each can be diffeomorphically mapped by charts to a bounded subset of~$\Reals^{d-1}$. Noting that these charts introduce volume distortions that are uniformly bounded and bounded away from zero, and applying \cite[Lemma~1.4.3]{Fri08}\footnote{This lemma gives the analytically computable exact value of the integral, if one replaces $f$ and $g$ by their respective bounds, and integrates with respect to~$\xi$ over~$\Reals^n$.} with~$n=d-1$ yields~\eqref{eq:convolution_bound}, where `$\const$' depends only on~$\partial\mathbb{X}$.
\end{proof}

\begin{proof}[Proof of Lemma~\ref{lem:kernel_bounds}]
Let~$\Gamma$ be a fundamental solution of the general problem~$\mathcal{L}f = 0$, i.e., $\Gamma(\cdot,\cdot; \xi,\tau)$, considered as a function in~$x,t$, satisfies~$\mathcal{L}\Gamma = 0$ for every~$\xi,\tau$, and 
\[
\lim_{t\downarrow \tau} \int_{\mathbb{X}} \Gamma(x,t;\xi,\tau) f(\xi)\,d\xi = f(x)\,,
\]
cf.~(1.2.8). Note that~$d$ denotes the dimension of our spatial domain~$\mathbb{X}$. Friedman shows (cf.\ (1.6.12)), that
\begin{equation} \label{eq:GammaBound}
|\Gamma(x,t;\xi,\tau)| \le \frac{\const}{(t-\tau)^{d/2}} \exp\left[ -\frac{\lambda_0^* | x-\xi |^2}{4(t-\tau)} \right]\,,
\end{equation}
where~$0<\lambda_0^*<\lambda_0$ is some constant, and~$\lambda_0$ depends only on~$L,\lambda_\sigma$ (from Assumption~\ref{ass:coefficient_bounds}), $\alpha$ and~$\mathbb{X}$; cf.~(1.2.2).

In what follows, we will establish a bound of this type for~$\Gamma_N$ as well. To this end, we will need a kind of an analogue of~\cite[Lemma~1.4.3]{Fri08}, with respect to integration on the boundary~$\partial\mathbb{X}$.

Now, let us construct the fundamental solution~$\Gamma_N$ of the problem with Neumann boundary conditions. The general form of the problem is stated in~\cite[(5.3.1)--(5.3.3)]{Fri08}, but note that his functions~$\beta(x,t), f(x,t), g(x,t)$ are identically zero in our case, such that we arrive at
\[
\begin{array}{rcll}
\mathcal{L}u &=&  0\quad &\text{in }\mathbb{X}\times (0,T], \\
u(\cdot,0) &=& f_0  &\text{on }\overline{\mathbb{X}}, \\
\frac{\partial u}{\partial\nu} &=& 0 &\text{on }\partial\mathbb{X}\times (0,T],
\end{array}
\]
where $\frac{\partial u}{\partial\nu}$ is the conormal derivative on~$\partial\mathbb{X}$. Proceeding as in~\cite[Theorem~5.3.2]{Fri08}, the solution of this problem is given by (cf.~(5.3.5), (5.3.6), and (5.3.8))
\begin{equation} \label{eq:Fri08_u}
u(x,t) = \int_0^t\int_{\partial\mathbb{X}} \Gamma(x,t;\xi,\tau)\varphi(\xi,\tau)\,dS_{\xi}d\tau + \int_{\mathbb{X}} \Gamma(x,t;\xi,0)f_0(\xi)\,d\xi,
\end{equation}
where
\begin{equation} \label{eq:Fri08_phi}
\varphi(x,t) = 2\int_0^t\int_{\partial\mathbb{X}} \frac{\partial \Gamma(x,t;\xi,\tau)}{\partial \nu(x,t)} \varphi(\xi,\tau)\,dS_{\xi}d\tau + 2F(x,t), \nonumber
\end{equation}
with
\begin{equation} \label{eq:Fri08_F}
F(x,t) = \int_{\mathbb{X}} \frac{\partial \Gamma(x,t;\xi,0)}{\partial \nu(x,t)} f_0(\xi)\,d\xi.
\end{equation}
We start with expressing~$\varphi$ explicitly via an infinite series, as in~(5.3.10). Setting $M_0 = 1 / \big( \int_0^t \int_{\partial\mathbb{X}} \, dS_{\xi}\,d\tau \big)$, $M_1(x,t; \xi,\tau) = \frac{\partial \Gamma (x,t; \xi,\tau)}{\partial \nu(x,t)}$, and for~$\ell \ge 1$ recursively
\begin{equation} \label{eq:M_recursion}
M_{\ell+1}(x,t; \xi,\tau) = \int_0^t\int_{\partial\mathbb{X}} M_1(x,t; y,\sigma) M_{\ell} (y,\sigma; \xi,\tau) \, dS_y\,d\sigma\,,
\end{equation}
it is shown on~\cite[Section~5.3, pp.~145]{Fri08} that~\eqref{eq:Fri08_phi} defines a continuous bounded function~$\varphi$ that is given by
\begin{equation} \label{eq:Fri08_phirecursion}
\varphi(x,t) = 2 \sum_{\ell=0}^{\infty} \int_0^t\int_{\partial\mathbb{X}} M_{\ell} (x,t; \xi,\tau) F(\xi,\tau) \, dS_{\xi}\,d\tau.
\end{equation}
We will now show that~$\varphi$ can be expressed in the kernel form
\begin{equation} \label{eq:phi_kernel_form}
\varphi(x,t) = \int_0^t\int_{\partial\mathbb{X}} \Psi (x,t; \xi,\tau) F(\xi,\tau) \, dS_{\xi}\,d\tau,
\end{equation}
where
\begin{equation} \label{eq:Psi_kernel}
\Psi(x,t; \xi,\tau) = 2  \sum_{\ell=0}^{\infty} M_{\ell} (x,t; \xi,\tau) ,
\end{equation}
and that~$\Psi$ satisfies a bound of type~\eqref{eq:GammaBound}. Following the derivation of~(5.2.12) we can show that for~$x\in \partial\mathbb{X}$ the bound
\begin{equation} \label{eq:GammaNuBound}
\left| \frac{\partial \Gamma(x,t;\xi,\tau)}{\partial \nu(x,t)} \right | \le \frac{\const}{(t-\tau)^{\mu}}\, \exp\left[ -\frac{\lambda_0^* | x-\xi |^2}{4(t-\tau)} \right]
\end{equation}
holds for any~$(1-\beta/2) < \mu$ and~$\lambda_0^* < \lambda_0$, where~$\beta = \min(\alpha,\delta)$. Using this bound with~$\mu = \frac{d+1-\beta}{2}$, we obtain from~\eqref{eq:M_recursion} with~Lemma~\ref{lem:Fri08_Lem1.4.3_analogue} that
\[
\left| M_2(x,t; \xi, \tau) \right| \le \frac{\const}{(t-\tau)^{d/2+1/2 - \beta}}\, \exp\left[ -\frac{\lambda_0^* | x-\xi |^2}{4(t-\tau)} \right],
\]
and further, by induction,
\[
\left| M_{\ell}(x,t; \xi, \tau) \right| \le \const \frac{ H_0 H^{\ell} }{\Gamma( \ell\beta )} (t-\tau)^{\ell\beta/2 - d/2 - 1/2 }\,  \exp\left[ -\frac{\lambda_0^* | x-\xi |^2}{4(t-\tau)} \right],
\]
where~$H_0,H$ are some positive constants, and~$\Gamma(\cdot)$ denotes the gamma function. This procedure is completely analogous to the steps in~\cite[Section~1.4, pp.~16]{Fri08} that lead to the bounds~(1.4.14) and~(1.4.15). Following Friedman's exposition, we conclude that the series~\eqref{eq:Psi_kernel} is absolutely convergent, that~\eqref{eq:phi_kernel_form} holds, and in particular, that
\begin{equation} \label{eq:PsiBound}
\left| \Psi(x,t,; \xi, \tau) \right| \le \frac{\const}{(t-\tau)^{(d+1 - \beta)/2}}\, \exp\left[ -\frac{\lambda_0^* | x-\xi |^2}{4(t-\tau)} \right].
\end{equation}
We are now ready to assemble the fundamental solution~$\Gamma_N$. Substituting~\eqref{eq:phi_kernel_form} and~\eqref{eq:Fri08_F} back into~\eqref{eq:Fri08_u}, the bounds~\eqref{eq:GammaBound}, \eqref{eq:GammaNuBound}, and~\eqref{eq:PsiBound} guarantee absolute integrability, and thus the upcoming integrals are interchangeable, giving
\begin{align}
u(x,t) &= \int_0^t\int_{\partial\mathbb{X}} \Gamma(x,t;\xi,\tau)\int_0^t\int_{\partial\mathbb{X}} \Psi(\xi,\tau; y,\sigma) \int_{\mathbb{X}} \frac{\partial\Gamma(y,\sigma; z,0)}{\partial\nu(y,\sigma)} f_0(z)\,dz \, dS_y d\sigma \, dS_{\xi}d\tau \nonumber \\
&\phantom{=} + \int_{\mathbb{X}} \Gamma(x,t; z,0)f_0(z)\,dz \nonumber \\
&= \int_{\mathbb{X}}  f_0(z) \int_0^t\int_{\partial\mathbb{X}} \underbrace{\frac{\partial\Gamma(y,\sigma; z,0)}{\partial\nu(y,\sigma)} \int_0^t\int_{\partial\mathbb{X}} \Psi(\xi,\tau; y,\sigma) \Gamma(x,t;\xi,\tau) \, dS_{\xi}d\tau }_{=: \Theta(y,\sigma; x,z,t)}\, dS_y d\sigma \, dz \label{eq:full_kernel_representation1}\\
&\phantom{=} + \int_{\mathbb{X}} \Gamma(x,t; z,0)f_0(z)\,dz. \label{eq:full_kernel_representation2}
\end{align}
Now, the bounds~\eqref{eq:GammaBound} and~\eqref{eq:PsiBound} with Lemma~\ref{lem:Fri08_Lem1.4.3_analogue} gives a bound for the innermost integral (the one with respect to $\xi,\tau$), i.e.,
\[
\int_0^t\int_{\partial\mathbb{X}} \left| \Psi(\xi,\tau; y,\sigma) \Gamma(x,t;\xi,\tau) \right| \, dS_{\xi}d\tau \le \frac{\const}{(t-\sigma)^{(d - \beta)/2}}\, \exp\left[ -\frac{\lambda_0^* | x-y |^2}{4(t-\sigma)} \right].
\]
With this,~\eqref{eq:GammaNuBound}, and Lemma~\ref{lem:Fri08_Lem1.4.3_analogue} again we obtain for the integral in~\eqref{eq:full_kernel_representation1} with respect to~$y,\sigma$ that
\begin{align*}
\int_0^t\int_{\partial\mathbb{X}} \big| \Theta(y,\sigma; x,z,t) \big| dS_y d\sigma \le \frac{\const}{t^{\mu^*}} \exp\left[ -\frac{\lambda_0^* | x-z |^2}{4 t} \right]
\end{align*}
for some~$\mu^* > 1/2$ arbitrary. With this bound and~\eqref{eq:GammaBound}, writing~\eqref{eq:full_kernel_representation1} and~\eqref{eq:full_kernel_representation2} in the form~$u(x,t) = \int_{\mathbb{X}} \Gamma_N(x,t; z, 0) f_0(z)\,dz$, we arrive at the desired bound
\begin{equation}
 \Gamma_N(x,t; \xi, 0) \le \frac{\const}{t^{d/2}} \exp\left[ -\frac{\lambda_0^* | x-\xi |^2}{4 t} \right], \nonumber
\end{equation}
which implies a uniform bound in~$x$ for any fixed~$t>0$.
\end{proof}

\begin{small}
\bibliographystyle{myalpha} 
\bibliography{bibliofile}
\end{small}
\end{document}